\documentclass[reqno,11pt]{amsart}
\usepackage[colorlinks=true, pdfstartview=FitV, linkcolor=blue, 
            citecolor=blue, urlcolor=blue]{hyperref}

\usepackage{graphicx}
\usepackage[mathscr]{euscript} 
\usepackage{bm}

\usepackage{amssymb}
\usepackage{amsmath}
\usepackage{amsthm}
\usepackage{pgf}
\usepackage{color}
\usepackage{graphicx}
\usepackage{comment}
\usepackage{mathtools}
\usepackage{hyperref}
\usepackage{cases}
\usepackage{enumitem}

\newcommand{\FFF}{\color{black}}

\newcommand{\PPP}{\color{black}}

\newcommand{\RRR}{\color{red}}
\newcommand{\BBB}{\color{black}}
\newcommand{\EEE}{\color{black}}
\newcommand{\UUU}{\color{black}}
\newcommand{\MMM}{\color{black}}
\newcommand{\OOO}{\color{black}}

\newcommand{\SSS}{\color{black}}
\newcommand{\CCC}{\color{black}}

\newcommand{\tone}{\boldsymbol t_1}
\newcommand{\ttwo}{\boldsymbol t_2}
\newcommand{\tthree}{\boldsymbol t_3}

\newtheorem{theorem}{Theorem}[section]
\newtheorem{lemma}[theorem]{Lemma}
\newtheorem{proposition}[theorem]{Proposition}

\theoremstyle{definition}
\begingroup
\newtheorem{definition}[theorem]{Definition}

\endgroup

\newcommand{\Qz}{\mathbb{Q}}
\newcommand{\Rz}{\mathbb{R}}
\newcommand{\Nz}{\mathbb{N}}
\newcommand{\Zz}{\mathbb{Z}}

\newcommand{\eps}{\varepsilon}

\newcommand{\be}[1]{\begin{equation}\label{#1}}
\newcommand{\ee}{\end{equation}}

\usepackage[bottom]{footmisc}

\begin{document}

%\title[Discrete to continuum passage  for the 2$d$ Winterbottom problem]{Discrete to continuum passage  for the 2$d$ Winterbottom problem}

\title[Microscopical justification of Winterbottom problem]{Microscopical Justification of the Winterbottom problem for well-separated Lattices} 
%distanced Lattices}
%well-distanced Lattices}
%nonoverlapping Lattices}

%\title[Discrete to continuum passage  for the 2$d$ Winterbottom problem]{Discrete to continuum passage  for the 2$d$ Winterbottom problem \RRR(Title to be probably changed)\BBB}

\author[Paolo Piovano]{Paolo Piovano}%*} 
%\thanks{*Corresponding author.}
\address[Paolo Piovano]{\CCC Department of Mathematics, Polytechnic University of Milan, P.zza Leonardo da Vinci 32, 20133 Milano, Italy \BBB}
%\address[Paolo Piovano]{Faculty of Mathematics, University of Vienna,   Oskar-Morgenstern-Platz 1, A-1090 Vienna, Austria}
%\address{\&}
%\address{Okinawa Institute of Science and Technology (OIST), 1919-1 Tancha, Onna-son, Kunigami-gun,  904-041919-1 Okinawa, Japan}
\email{paolo.piovano@polimi.it}

\author{Igor Vel\v{c}i\'c}
\address[Igor Vel\v{c}i\'c]{Faculty of Electrical Engineering and Computing, University of Zagreb, Unska 3, 10000 Zagreb, Croatia}
\email{igor.velcic@fer.hr}


%Moving forward from the specific discrete atomistic setting introduced in a previous paper by the authors to microscopically justify  such model, we relax the rigidity assumption considered in  that paper to characterize the \emph{wetting} and \emph{dewetting regimes} and to perform the \emph{discrete to continuum passage}.
\begin{abstract}
We consider the discrete atomistic setting introduced in  \cite{PiVe1}  to microscopically  justify  the continuum model related to the \emph{Winterbottom problem}, i.e., the problem of determining the equilibrium shape of crystalline film drops resting on a substrate, and relax the rigidity assumption considered in  \cite{PiVe1} to characterize the \emph{wetting} and \emph{dewetting regimes} and to perform the \emph{discrete to continuum passage}. In particular, all results of  \cite{PiVe1} are  extended  to the setting where the distance between the  reference lattices for the film and the substrate is not smaller than the optimal bond length between a film and a substrate atom. Such optimal film-substrate bonding distance is prescribed together with the optimal film-film distance by means of two-body atomistic interaction potentials of Heitmann-Radin type, which are both taken into account in the discrete energy, and in terms of which the wetting-regime threshold and the effective expression for the wetting parameter in the continuum energy are determined.

 \end{abstract}

\maketitle

\pagestyle{myheadings}
%%%%%%%%%%%%%%%%%%%%%%%%%%%%%%%%%%%%%%%%%%%%%%%%%%

\section{Introduction}

In this manuscript we address the classical problem of  determining the equilibrium shape formed by \FFF crystalline film drops \EEE resting upon rigid substrates possibly of a different material. Such problem has a variational nature as the solution can be looked for among the minimizers of a surface energy that depends both on the  \emph{drop anisotropy} at the drop surface   and on the \emph{drop wettability} at the contact region with the flat substrate. By exploiting the interplay between the drop anisotropy and wettability W.\ L.\ Winterbottom
provided in \cite{Winterbottom}  the first phenomenological prediction of the solution  by a direct construction of  the nowadays called  \emph{Winterbottom shape}   (see Figure \ref{fig: winterbottom}). %as the minimizer of a surface energy depending both on the  \emph{drop anisotropy} at the drop surface   and on the \emph{drop wettability} at the contact region with the flat substrate, the variational nature  of the problem. 
We intend here to move forward from the results obtained by the authors in \cite{PiVe1} by extending 
 the discrete to continuum derivation of the energy considered by Winterbottom (see also  \cite{srolovitz1,srolovitz2}) established in \cite{PiVe1}  to all mutual positionings of the reference lattices  of the film and of the substrate, whose distance is not smaller than the optimal bond distance between a film and a substrate atom.

More precisely, the energy considered by W.\ L.\ Winterbottom  in the continuum planar setting in  \cite{srolovitz1,srolovitz2,Winterbottom} expressed in modern terminology  is defined  for any set of finite perimeter  $D\subset\Rz^2\setminus S$ representing the area occupied by a film drop outside of the \emph{substrate region} $S\subset\Rz^2$ by 
\begin{equation}\label{winterbottomsurfaceenergy} 
 \mathcal{E}(D):=\int_{\partial^* D\setminus\partial S}\Gamma(\nu(\xi))\,  \mathrm{d}\mathcal{H}^{1}(\xi) +  \sigma \mathcal{H}^{1}(\partial^* D\cap\partial S),% \mathrm{d}\sigma_{\xi}
\end{equation}
where  $\Gamma:\mathbb{S}^{1} \to \mathbb{R}$ is the \emph{anisotropic surface tension} related to the material of the crystalline drop and defined on the normal $\nu\in\mathbb{S}^{1}$ at the reduced boundary $\partial^* D$,  while $\sigma\in\Rz$ is a parameter representing the \emph{wetting},  namely, the ability  of the drop to maintain contact with the specific solid surface $\partial S$, which is the topological boundary of $S$, and $\mathcal{H}^{1}$ is the 1-dimensional Hausdorff measure (see \ref{sec:continuum_model} for more details). By adapting the phenomenological construction provided by G. Wulff  in  \cite{Wulff01} of the set, nowadays named after him, for the related problem of finding  the equilibrium shape of a free-standing crystal with anisotropy $\Gamma$ in the space, i.e., for the minimizer of \eqref{winterbottomsurfaceenergy} in the case with $S=\emptyset$ (see also \cite{F,FM2}),  that is
$$
W_\Gamma:=\{x\in\Rz^d\,:\, x\cdot \nu\leq \Gamma(\nu) \text{ for every } \nu\in S^{d-1}\}, 
$$ 
 the Winterbottom shape for the flat substrate $S=\Rz\times\{s\in\Rz\,:\, s<0\}$ is the set 
$$W_{\Gamma,\sigma}:=W_\Gamma\cap\{x\in\Rz^d\,:\, x_d\geq - \sigma\},$$
 as depicted in Figure \ref{fig: winterbottom}. 
\begin{figure}
\includegraphics[width=0.49\textwidth]{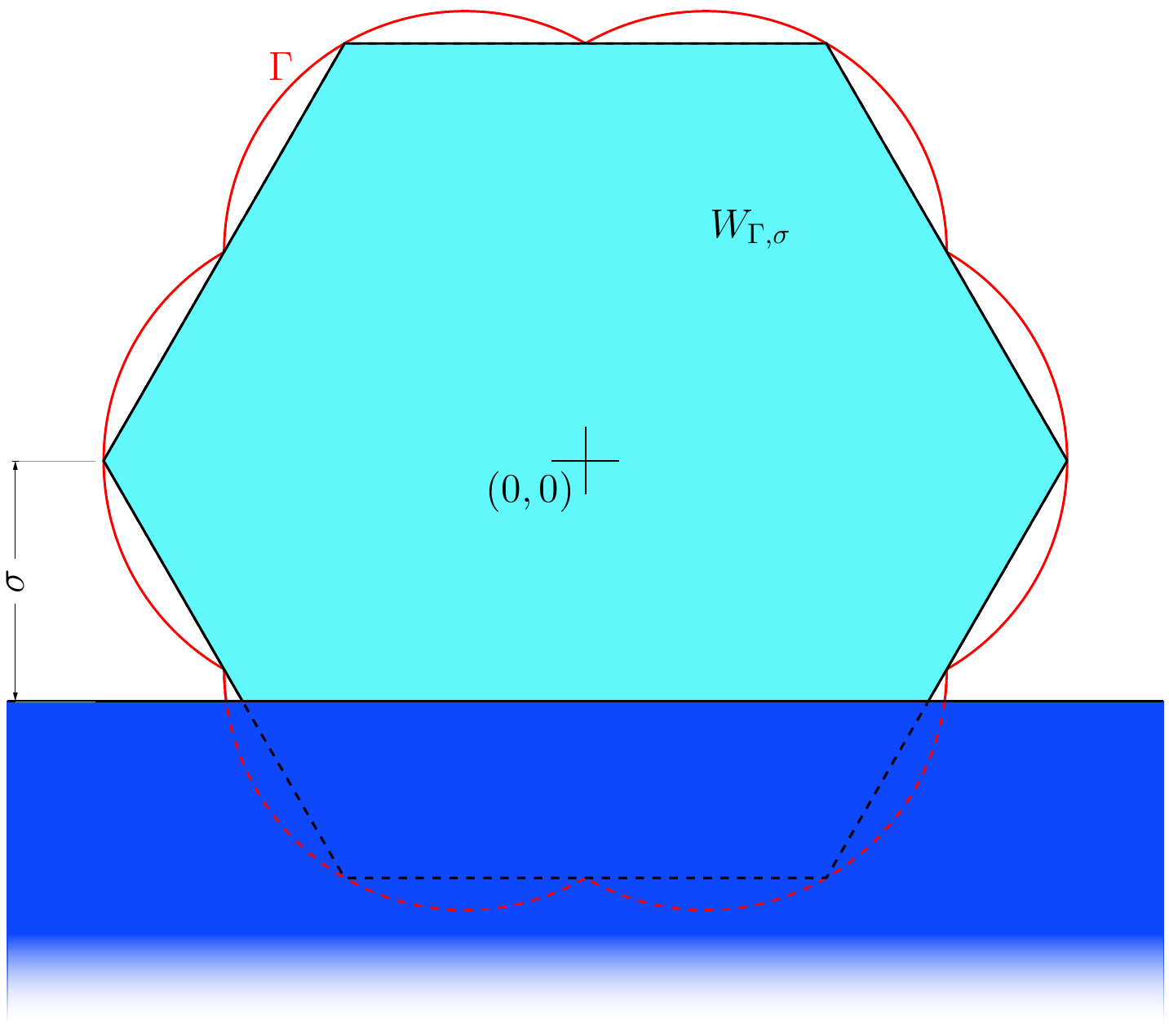}
\includegraphics[width=0.49\textwidth]{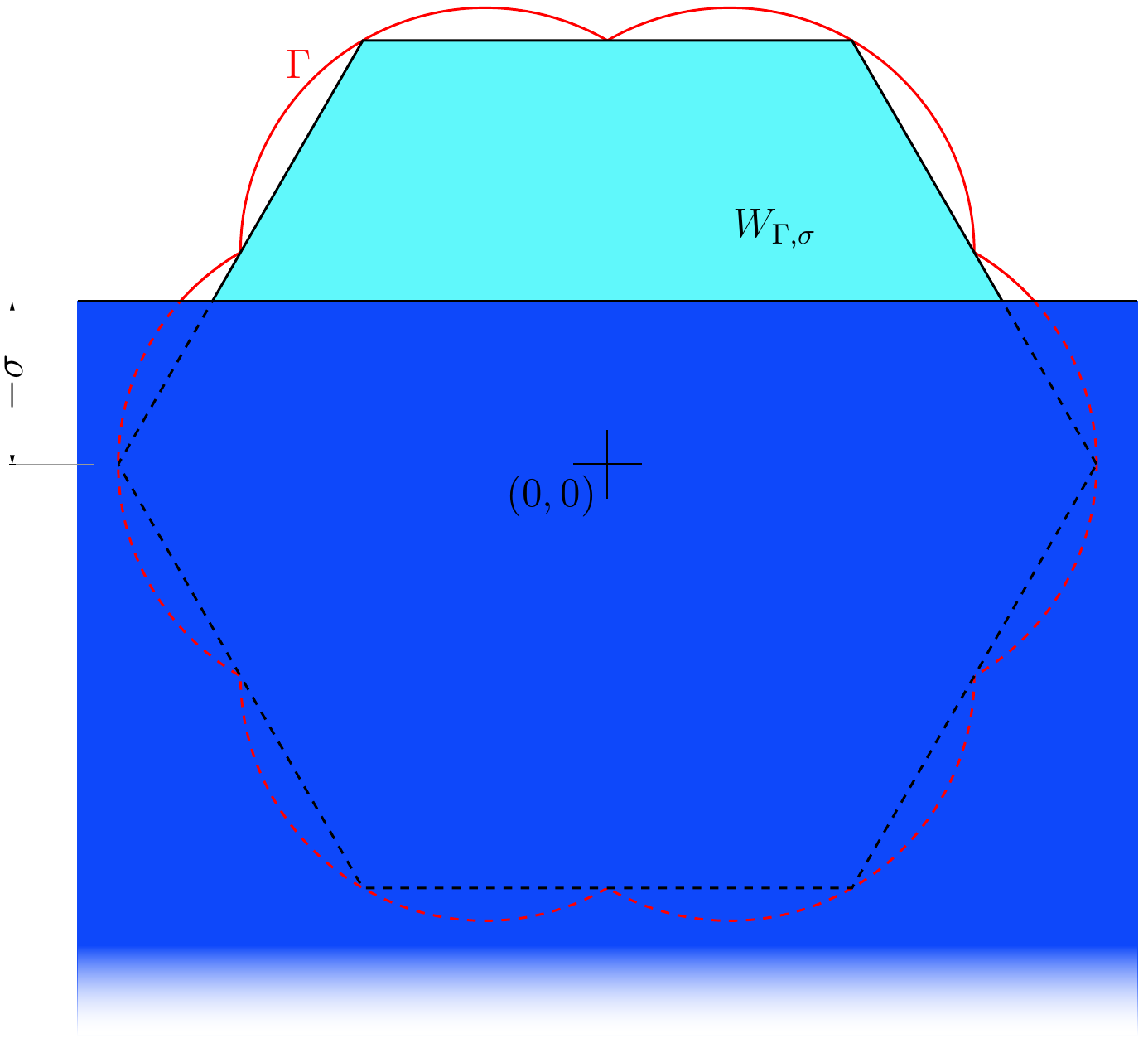}
\caption{\PPP Winterbottom construction for the minimizer of $ \mathcal{E}$, \PPP on the left for $\sigma>0$ and on the right for $\sigma<0$ \EEE (see  \cite{Winterbottom}). }
\label{fig: winterbottom}
\end{figure}

For the justification of the problem in dimension $d=2$ in the context of statistical mechanics and the \emph{Ising model} we refer to 
the review  \cite{DKS} (see also \cite{ISc,KP})  for the Wulff shape in the scaling limit at low-temperature and  to \cite{Dodineau-etal,PV1,PV2} for the setting related to the Winterbottom shape, while  in the context of atomistic models a rigorous discrete to continuum passage  for triangular reference lattices has been first carried out by means of $\Gamma$-convergence in \cite{Yeung-et-al12} for the Wulff shape and then extended to the Winterbottom situation in \cite{PiVe1}. The emergence of the Wulff shape has been also deduced for the square lattice in \cite{MPS,MPS2} and the hexagonal lattice  \cite{DPS2}   by means of a different approach based on induction techniques related to the crystallization problem  \cite{Heitmann-Radin80}, and of the quantification of the deviation of discrete ground states from the asymptotic Wulff shape in the so called $n^{3/4}$ law (where $n$ is the number of atoms), which was previously introduced in \cite{Schmidt}, and then extended to those settings  (see also \cite{DPS1}), and more recently to higher dimensions in  \cite{MPSS,MS}. A derivation by $\Gamma$-convergence  of an energy of the type \eqref{winterbottomsurfaceenergy} coupled with a  bulk elastic term in  the context of models for epitaxially-strained thin films introduced in \cite{DP1,DP2,FFLM2, S2,spencer1997equilibrium} has been instead determined in  \cite{KrP} under a \emph{graph constraint} for the region occupied by the film drop.

We follow the approach of  \cite{PiVe1}, where the authors introduced a specific discrete setting, which here we intend to generalize, to initiate the analysis for the Winterbottom problem in the context of atomistic models by taking into consideration not only atomistic interactions among each other film atoms as in \cite{Yeung-et-al12}, but also between film and substrate atoms, in terms of two-body potentials $v_{F,\alpha}$ of Heitmann-Radin sticky-disc  type \cite{Heitmann-Radin80}  with $\alpha=F,S$ for film-film and film-substrate atomistic interactions, respectively. More precisely, %references lattices for the substrate and the film atoms are considered, i.e., $\mathcal{L}_S\subset S$ and $\mathcal{L}_F\subset\mathbb{R}^2\setminus \overline{S}$, respectively, and  a reference lattice $\mathcal{L}_S\subset S$ is assumed to be fully occupied by substrate atoms while 
in   \cite{PiVe1} the discrete setting involves a reference lattice $\mathcal{L}_S\subset \overline{S}:=\Rz\times\{s\in\Rz\,:\, s\leq0\}$ for the substrate atoms, which  is assumed to be fully occupied, and a triangular reference lattice $\mathcal{L}_F\subset\mathbb{R}^2\setminus \overline{S}$ defined by
$$
\mathcal{L}_F:=\{x_F+k_1 \tone+k_2 \ttwo\,:\, \UUU \textrm{$k_1\in\Zz$ and $k_2\in\Nz\cup\{0\}$} \BBB\}%\quad\textrm{for $\tone:=\left(1,0\right)$ and $\ttwo:=\left(\frac{1}{2},\frac{1}{2\sqrt{3}}\right)$}.
$$
for a fixed $x_F\in\Rz^2\setminus\overline{S}$, which we call \emph{film-lattice center},
$$\tone:={1 \choose 0},\quad\textrm{and}\quad\ttwo:= \UUU  \frac{1}{2}{1 \choose \sqrt{3}},
$$
%which we refer to as the \emph{film-lattice center}.
\noindent on which film atoms are free to choose the most convenient configuration under the considered atomistic interactions $v_{F,\alpha}$ with $\alpha=F,S$. Thus,  for each fixed number $n\in\mathbb{N}$  all sets $D_n:=\{x_1,\ldots,x_n\}\subset\mathcal{L}_F$ are considered admissible configurations of $n$ film atoms (see Figure \ref{fig:lattices}), and the overall energy $V_n: (\mathbb{R}^2\setminus \overline{S})^{n}\to\Rz\cup\{\infty\}$ of a configuration $D_n:=\{x_1,\ldots,x_n\}\subset\mathcal{L}_F$ is 
defined by
$$
V_n(D_n)=V_n(x_1,\ldots,x_n):=\sum_{i\neq j} v_{FF}(|x_i-x_j|)\,+\,\sum_{i=1}^n \sum_{s\in\partial\mathcal{L}_S} v_\textrm{FS}(|x_i-s|),
$$
where the sums are extended to nearest neighbors and $\partial\mathcal{L}_S:=\mathcal{L}_S\cap\Rz\times\{s\in\Rz\,:\, s=0\}$ is referred to as the \emph{substrate surface} or \emph{wall}. Notice that in \cite{PiVe1} the minimum of the atomistic potentials $v_{F\alpha}$ is fixed at $e_{F\alpha}>0$  with values $-c_{\alpha}<0$ for $\alpha=F,S$, respectively. where $e_{FF}$ is normalized at 1 and coincide with the lattice parameter of  $\mathcal{L}_F$, while the following \emph{rigidity assumption} was made in regard of $e_{FS}$: 
\begin{equation}\label{rigidity_assumption} 
\text{$x_F:=x_F^0$\quad  with\quad  $x_F^0:=(0,e_{FS})$},
\end{equation}
which in particular by the choice of $\partial\mathcal{L}_S$ was  preventing the possibility of two substrate neighbors for film atoms (which is the maximum possible number of substrate neighbors  for the situation with a flat substrate and a  Heitmann-Radin potential  $v_{FS}$). %restricting the maximum numbers of substrate atoms  also that film atoms can have at most one substrate atoms.

The aim of this manuscript is to relax the rigidity assumption \eqref{rigidity_assumption} and obtain the full generality of the results in \cite{PiVe1} for all the positionings $x_F\in\Rz^2\setminus\overline{S}$ for $\mathcal{L}_F$ 
  for which 
\begin{equation}\label{firstinterface_intro}
{\rm dist}(\partial \mathcal{L}_F,\partial \mathcal{L}_S)\geq e_{FS}  %\RRR \quad  {\rm dist}(\mathcal{L}_F\setminus\partial\mathcal{L}_F,\partial\mathcal{L}_S)>e_{FS}, \BBB
\end{equation}
(which include \eqref{rigidity_assumption}) and hence, in particular allowing film atoms (resp. substrate atoms) to display from zero to two substrate  neighbors (resp. film neighbors).

The results are threefold (see Section \ref{sec: main_result}): First, we characterize in Theorem \ref{wetting_theorem} the wetting regime, i.e., when film atoms are expected  to spread on the substrate surface instead of accumulating in island clusters on top of it, in terms of the parameters $c_{\alpha}$  related to the atomistic potentials $v_{F\alpha}$ for $\alpha=F,S$. Furthermore, also the corresponding minimizers in such regime are explicitly isolated for every $n\in\Nz$ by  induction techniques related to crystallization problems (see, e.g., \cite{Heitmann-Radin80,MPS}). In particular, as a result of the much more involved setting  an extra  threshold condition  under which the film spreads in an infinitesimally thick layer is found for certain settings in between the thresholds already determined in \cite{PiVe1}. 

Second,  we prove in Theorem \ref{connectness} that in the dewetting regime, where  film atoms are expected to form solid-state islands (related to regions with positive two-dimensional Lebesgue measure $\mathcal{L}^2$),  the mass of the  solutions of the discrete minimum problems
\begin{equation}\label{discretemodels}
\min_{D_n\subset\mathcal{L}_F}{V_n(D_n)}\
\end{equation}
as the number $n$ of atoms tends to infinity does not spread, but up to a subsequence it is preserved by the connected components of such minimizers  with largest cardinality. This is crucial to overcome the lack of compactness that we have, as already detailed in \cite{PiVe1}, outside the class of \emph{almost-connected configurations}, which are roughly speaking, configurations connected up to a substrate-bond distance (see Section \ref{transformation}  for more details). 

 Finally, we establish in Theorem \ref{thm:convergence_minimizers} that in the dewetting regime the solutions of the discrete problems \eqref{discretemodels} converge as the number $n$ of  film atoms tends to infinity  (up to extracting a subsequence and performing horizontal translations on the substrate $S$) to a minimizer of the Winterbottom energy \eqref{winterbottomsurfaceenergy}   in the family of crystalline-drop regions 
$$
\mathcal{D}_\rho:=\{D\subset\Rz^2\setminus S\,:\, \text{set   of finite  perimeter, bounded and such that  $|D|=1/\rho$}\}, 
$$
where $\rho$ is the atom density  in $\mathcal{L}_F$ per unit area. This convergence of the discrete minimizers in the dewetting regime  is obtained in view of the conservation of mass proven of Theorem \ref{connectness} by following the approach in \cite{Yeung-et-al12,PiVe1} and proving the $\Gamma$-convergence as $n\to\infty$ of properly defined (and rescaled) versions of $V_n$  in the space $\mathcal{M}(\Rz^2)$ of Radon measures  on $\Rz^2$ % , i.e., $\mathcal{M}(\Rz^2)$, 
with respect to the weak* convergence of measures. In particular, an effective expression for the wetting parameter $\sigma$ in \eqref{winterbottomsurfaceenergy}  is obtained, i.e., 
\begin{equation}\label{intro_sigma}
\sigma:=\begin{cases}
2c_F-\displaystyle\frac{2c_S}{q} & \text{for $C_i$ with $i=1,3,4$,}\\ 
&\\
2c_F-\displaystyle\frac{c_S}{q}  & \text{for $C_2$,}\\
   \end{cases}
\end{equation}
where  $q\in\Nz$ relates to the optimal substrate bond  $e_S:=q/p$ in $\partial\mathcal{L}_S$ with  $p\in\Nz$, and $C_i$ with $i=1,\dots,4$ are categories in which the various settings allowed by \eqref{firstinterface_intro} can be classified (see Section \ref{sec: main_result} for more details). 

Our methodology consists in introducing a more general setting (even more general of the settings described above) with substrate wall
 \begin{equation}\label{wall_intro}
\partial\mathcal{L}_S=\left\{s_k:=\left(\frac{k}{p},0\right) \,:\,k\in q\Zz\cup\left(q\Zz+r\right)\right\}
\end{equation}
%$$e_S:=\min_{x,x'\in\partial\mathcal{L}_S}{|x-x'|}>0, \ZZZ \text{(I thgink we could take away this assumption and allowing  the inf to be zero?)}\PPP$$
for \emph{wall parameters} $z:=(p,q,r)\in Z_S$ with
$$Z_S:=\left\{(p,q,r)\in \Nz\times\Nz\times\Nz_0 : \text{$p$ and $q$ co-prime, $0\leq r\leq\frac{q}{2}$, and $p=1$ if $r\neq0$}\right\}$$
where  $\Nz_0:= \Nz\cup\{0\}$, which reduces to the relevant settings  described above for 
$$z\in Z_S^0:=\{(p,q,r)\in Z_S : r=0\}$$
 (and in particular to the model of  \cite{PiVe1} for $x_F:=x_F^0$ and $z\in Z_S^0$) and to an extra \emph{auxiliary setting} for $z\in Z_S^1:=Z_S\setminus Z_S^0$. Such an auxiliary setting is carefully determined in order to both be able to implement all the program of  \cite{PiVe1} for it as well, and as the only extra needed  model for treating all the settings with $z\in Z_S^0$ that cannot be reduced to the model in  \cite{PiVe1}. More precisely, we introduce an \emph{equivalence relation} among the various possible settings with substrate wall \eqref{wall_intro} and prove that all the settings with $z\in Z_S^0$ satisfying \eqref{firstinterface_intro} can be reduced either to the model of  \cite{PiVe1} or to the auxiliary model  with $z\in Z_S^1$. Since in particular such equivalence relation preserves all the properties contained in the main results, we are then allowed to transfer such properties from  the model already treated in  \cite{PiVe1} and from the auxiliary model, by directly proving them extra only for the latter.  
 
 %Finally, we observe that as a byproduct of our strategy analogous results as Theorems \ref{wetting_theorem}, \ref{connectness}, and \ref{thm:convergence_minimizers} are also obtained for any models $\mathcal{M}^1_{\Lambda}(z)$ with $z\in Z^1_S$ and $\Lambda:=(e_{FS},c_F,c_S)\in (\Rz^+)^3$, namely in Theorems \ref{wetting_theorem_Z1}, \ref{connectness_Z1}, and \ref{thm:convergence_minimizers_Z1}, respectively.

More specifically,  the classes $C_i$ with $i=1,\dots,4$  for the various settings with $z\in Z_S^0$  are exactly introduced to prove the equivalence to the two specific setting: the settings in classes  $C_i$ for $i=1,\dots,3$ are equivalent to the model in  \cite{PiVe1} (with different lattice parameters) and the settings in the class $C_4$ to the auxiliary model introduced in this manuscript. Moreover, a specific feature of the auxiliary model for $z\in Z_S^1$ is the possibility of having separated pairs of film atoms bonded with the substrate. To accommodate such aspect in the strategy of \cite{PiVe1} we need to include the extra wetting condition appearing in Theorem \ref{wetting_theorem}, extend the strip argument used to establish the compactness for almost-connected configurations used to prove the conservation of mass of Theorem \ref{connectness} (see the more involved definition for the \emph{strip energy} in Section \ref{local_energy}), and finally adjust the ``boundary-averaging'' arguments used for the lower bound in the proof of the $\Gamma$-convergence, which in turns is responsible for the more involved form of the effective expression for the wettability $\sigma$ in \eqref{intro_sigma} with respect to  \cite{PiVe1}.

%\PPP \subsection{\PPP Local and strip energies} \label{local_energy}  
 
The paper is organized as follows. In Section \ref{setting} we introduce the mathematical setting and the main results of the paper. In Section \ref{sec:M1}  we implement the program of  \cite{PiVe1} for the auxiliary model with substrate wall $\partial\mathcal{L}_S$ of the type \eqref{wall_intro} with $z\in Z_S^1$. In Section \ref{sec:main_results}  we prove the equivalence property for each class $C_i$ with $i=1,\dots,4$ of settings satisfying  \eqref{firstinterface_intro} and the main results of the paper. 
 
\newpage

\section{Mathematical setting \PPP and main results} \label{setting}
\PPP
In this section we introduce the discrete and continuous models, the notation and definitions  used throughout the paper, and the main results. 
%\FFF In the following  for a set $A \subset \mathbb{R}^d$ we denote by   $\partial A$ its topological boundary and by $\overline{A}$ its closure. We denote by $B(x,R)$ an open ball of radius $R$ centered at $x \in \mathbb{R}^2$ and $B(R):=B(o,R)$  where $o$is the origin in $\mathbb{R}^2$. \EEE

 %By $B(x,R)$ we denote the open ball of radius $R$ around point $x \in \mathbb{R}^2$. $\bar{B}(x,R)$ denotes the closed ball. $B(R)$ denotes the ball of radius $R$ with center at the origin while $\bar{B}(R)$ its closure. For a set $A \subset \mathbb{R}^2$, $\partial A$ denotes its topological boundary, $\overline{A}$ its closure, 
 
% $\# A$ its cardinality, 
%$\textrm{diam} (A)$ denotes its diameter.
 %If $A \subset \mathbb{R}^2$ is a set of finite perimeter $\partial^* A$ denotes its essential boundary. 
%For $A,B \subset\Rz^2$ by $\textrm{dist}(A,B)$ we denote the distance between sets $A$ and $B$.  

%For an open set $A \subset \Rz^n$ by $\mathcal{B}(A)$ we denote the Borel $\sigma$-algebra on $A$. 
%$A \subset\subset B$ means that the closure of $A$ is contained in $B$. 
%For a point $x\in \Rz^2$, $\delta_x$ denotes Dirac measure concetrated at $x$. 
%By $C_0(\Rz^2)$ we denote the set of continuous functions with compact support in $\Rz^2$.  

\subsection{\PPP Setting with lattice configurations}\label{sec:lattice_configurations}   We begin by introducing a reference \SSS set \BBB  $\mathcal{L}$  for the atoms of the substrate and of the film in  the plane $\Rz^2$ for  a chosen cartesian coordinate system. %We refer in the following to $S:=\Rz\times\{r\in\Rz\,:\, r<0\}$ as to the substrate region. 
We define  $\mathcal{L}:=\mathcal{L}_S\cup\mathcal{L}_F$, where $\mathcal{L}_S\subset\overline{S}$ denotes the reference lattice for the substrate atoms, with $S:=\Rz\times\{s\in\Rz\,:\, s<0\}$  referred to as the  \emph{substrate region},  and  $\mathcal{L}_F\subset\Rz^2\setminus\overline{S}$ is the reference lattice for the film atoms. %by do not interpenetrate and are arranged in a  lattice $\mathcal{L}:=\mathcal{L}_S\cup\mathcal{L}_F$, where $\mathcal{L}_S\subset\Rz\times\{r\in\Rz\,:\, r\leq0\}$ is the reference lattice for the substrate atoms and  $\mathcal{L}_F\subset\Rz\times\{r\in\Rz\,:\, r>0\}$ is the reference lattice for the film atoms. 

More precisely, we consider the substrate lattice as a fixed set, i.e., every lattice site in $\mathcal{L}_S$ is occupied by a substrate atom,  such that 
\begin{equation}\label{wall}
\partial\mathcal{L}_S:=\mathcal{L}_S\cap\{(s,0):\,s\in\Rz\,\}=\left\{s_k:=\left(\frac{k}{p},0\right) \,:\,k\in q\Zz\cup\left(q\Zz+r\right)\right\}
\end{equation}
%$$e_S:=\min_{x,x'\in\partial\mathcal{L}_S}{|x-x'|}>0, \ZZZ \text{(I thgink we could take away this assumption and allowing  the inf to be zero?)}\PPP$$
for $z:=(p,q,r)\in Z_S$ with
$$Z_S:=\left\{(p,q,r)\in \Nz\times\Nz\times\Nz_0 : \text{$p$ and $q$ co-prime, $0\leq r\leq\frac{q}{2}$, and $p=1$ if $r\neq0$}\right\}$$
where  $ \Nz_0:= \Nz\cup\{0\}$.  We refer to $\partial\mathcal{L}_S$ as to the  \emph{substrate wall} (or \emph{substrate surface}) and the vectors $z=(p,q,r)\in Z_S$  as the \emph{wall vector} with \emph{wall parameters} $p$, $q$, and $r$. We notice that we choose such a definition for $\partial\mathcal{L}_S$ even though it could be simplified, in order to directly include the setting used in \cite{PiVe1} without inconsistent notation between the two papers. Furthermore, notice that the choice of wall vectors $z=(1,q,r)$ with $q/2<r\leq q$ are excluded in the definition of $Z_S$ without loss of generality up to a translation of the origin  in the $\tone$ direction. 

For the film lattice   $\mathcal{L}_F$ we choose a triangular lattice with parameter  $e_F$ normalized to 1 (with respect to the wall parameters), namely 
\begin{equation}\label{LFlattice}
\mathcal{L}_F:=\{x_F+k_1 \tone+k_2 \ttwo\,:\, \UUU \textrm{$k_1\in\Zz$ and $k_2\in\Nz\cup\{0\}$} \BBB\}%\quad\textrm{for $\tone:=\left(1,0\right)$ and $\ttwo:=\left(\frac{1}{2},\frac{1}{2\sqrt{3}}\right)$}.
\end{equation}
for 
$$\tone:={1 \choose 0},\quad\textrm{and}\quad\ttwo:= \UUU  \frac{1}{2}{1 \choose \sqrt{3}}, % \left({\begin{array}{c}
 %1\\\\
%  \sqrt{3}\\
% \end{array}}\right)\EEE
$$%$x_F\in\Rz^2\setminus\overline{S}$,  \SSS 
and
$$x_F:=(x_F^1,x_F^2)\in\Rz^2\setminus\overline{S},$$ 
which we refer to as the \emph{film-lattice center}. Let also $\tthree$ be 
$$
\tthree:={0 \choose 1}.
$$
%{\color{pink} With the analysis presented here it is also possible to assume different values of $x_F$. } 
%We notice that we can avoid the setting with $z\in \Nz\times\Nz$  and $z_1/2<z_2\leq z_1$ without loss of generality up to a translation of the origin in the $\tone$ direction. 

%\CCC We recall that the setting with $z_2=0$ or $z_1=2$ and $z_2=1$ has been already studied in \cite{PiVe2}.

%We refer to Section \ref{sec:discussion}  for examples of other positioning of the reference lattices $\mathcal{L}_F$ and $\mathcal{L}_S$, which can be reduced to the one addressed in this mathematical setting. \BBB 
 The sites of the film lattice are not assumed to be completely filled and we refer to a set of $n\in\Nz$ sites $x_1,\dots,x_n\in \mathcal{L}_F$ occupied by film atoms as a \emph{crystalline configuration} denoted by $D_n:=\{x_1,\dots,x_n\}\subset \mathcal{L}_F$. Notice that the labels for  the elements of a configuration $D_n$ are uniquely determined by increasingly assigning them with respect to a chosen fixed order \FFF on \EEE the lattice sites \FFF of \EEE $\mathcal{L}_F$.  With a slight abuse of notation we refer to $x\in D_n$ as an atom in $D_n$ (or in $\mathcal{L}_F$). We denote the family of crystalline configurations with $n$ atoms by $\mathcal{C}_n$. \FFF Furthermore, given a set $A\subset \mathbb{R}^2$, its cardinality is indicated by $\#A$, so that $$\mathcal{C}_n:=\{A\subset\mathcal{L}_F\,:\, \#A=n\}.$$ \EEE

For every atom $x\in\mathcal{L}_F$ we take into account both its atomistic interactions with other film atoms and with the substrate atoms, by considering the two-body atomistic potentials \SSS $v_{FF}$ \BBB and \SSS$v_{FS}$\BBB, respectively. We restrict to first-neighbor interactions and we define \SSS $v_{F\alpha}$ \BBB for $\alpha:=F,S$ as 
\SSS \begin{equation} \label{defvf} 
v_{F\alpha}(s):=\begin{cases} +\infty &\mbox{if }  s<e_{F\alpha},\\ 
-c_{\alpha} &\mbox{if }  s=e_{F\alpha},\\ 
0 & \mbox{if }  s>e_{F\alpha},
\end{cases}
\end{equation} 
with  $c_{\alpha}>0$, $e_{FF}:=e_F=1$, and $e_{FS}>0$.
\BBB

We  denote by $\partial \mathcal{L}_F$ the \emph{lower boundary} of the film lattice, i.e.,
$$
\partial\mathcal{L}_F:=\{x_F+k_1 \tone:\, \UUU \textrm{$k_1\in\Zz$} \BBB\}
$$
\PPP and by  $\partial \mathcal{L}_{FS}$ the collection of sites  in the lower boundary of the film lattice at a distance of \CCC $e_{FS}$ \BBB from an atom in $\partial\mathcal{L}_S$, i.e., 
$$\partial \mathcal{L}_{FS}:= \partial \mathcal{L}_F \cap \{x\in\Rz^2 : {\rm dist}(x,\partial\mathcal{L}_S)=e_{FS}\}.$$

 In the following, we refer to  \emph{film} and \emph{substrate neighbors}   of an atom $x$ in a configuration $D_n$ as to those atoms   in  $D_n$ at distance 1 from $x$, and to those atoms in  $\mathcal{L}_S$ at distance $e_{FS}$  from $x$, respectively. Analogously, we refer to  \emph{film} and \emph{substrate bonds} of an atom $x$ in a configuration $D_n$ as to those segments connecting $x$ to its film and substrate neighbors, respectively.  More generally the same terminology will be extended for sites of $\mathcal{L}_F$. % other atoms (if any)  in  $D_n$ at distance 1 from $x$, or to those atoms in  $\mathcal{L}_S$ at distance $e_S$. and to those segment atoms in $\mathcal{L}_S$ and film, resp. substrate,  \emph{bonds}  as to the other atoms in  and 
We also refer to the union of the closures of all film bonds of atoms in a configuration $D_n$ as the \emph{bonding graph} of  $D_n$, and we say that a crystalline configuration $D_n$ is \emph{connected} if every $x$ and $y$ in $D_n$ are connected through a path in the bonding graph of $D_n$, i.e., there exist $\ell\leq n$ and $x_k\in D_n$ for $k:=1,\dots,\ell$ such that $|x_k-x_{k-1}|=1$, $x_1=x$, and $x_\ell=y$. \FFF Moreover,  we define  the \emph{boundary of a configuration} $D_n\in\mathcal{C}_n$ as the set  $\partial D_n$ of atoms of $D_n$ with less than 6 film neighbors. \FFF We notice here that with a slight abuse of notation, given a  set $A\subset\mathbb{R}^2$  the notation $\partial A$ will also denote the topological boundary of a set $A\subset\mathbb{R}^2$  (which we intend to be always the way to interpret the notation when applied not to configurations in $\mathcal{C}_n$, or to lattices, such as for $\partial \mathcal{L}_{S}$, $\partial \mathcal{L}_{F}$, and $\partial \mathcal{L}_{FS}$). \EEE 

\EEE

% In the following we will make use of the following notion of  configurations connected through paths along the bonding graph.

%\begin{definition}\label{connected_configuration}
% We say that a configuration $D_n$ is \emph{connected} if for every $x_i$ and $x_j$ in $D_n$ there exists $y_k\in D_n$ for $k:=1,\dots,\ell$ with $\ell\leq n$ such that $|y_k-y_{k-1}|=1$, $y_0=x_i$, and $y_n=x_j$.
% \end{definition}

The energy $V_n$ of a configuration $D_n:=\{x_1,\ldots,x_n\}\subset \mathcal{L}_F$ of $n$ particles is defined by 
%Our aim is to provide a \emph{microscopic justification} of the Winterbottom shape starting from energies $E_n:\UUU(\mathbb{R}\times\mathbb{R}^+)^{n}\EEE\to\mathbb{R}\cup\{+\infty\}$ defined on configurations $D_n:=\{x_1,\ldots,x_n\}\subset \UUU(\mathbb{R}\times\mathbb{R}^+)^{n}\EEE$ of the $n$ droplet atoms by 
\begin{equation}\label{V}
\UUU V_n(D_n)=\EEE V_n(x_1,\ldots,x_n)\PPP:=\UUU\sum_{i\neq j} v_{FF}(|x_i-x_j|)\,+\,  E_S(x_1,\ldots,x_n)\EEE%\,+\,\frac{1}{2}\sum_{(i,j,k)\in \mathcal{A}} v_3(\theta_{ijk})\,,\qquad\,  \raisebox{-10pt}{\includegraphics[scale=0.7]{OSCI_Figures/bonds}}
\end{equation}
where  $E_S:(\Rz^2\setminus\overline{S})^{n}\EEE\to\Rz\cup\{\infty\}$ represents the overall contribution of the substrate interactions defined  as  
\begin{equation}\label{substrateenergy}
E_S(D_n)=E_S(x_1,\dots,x_n):=\sum_{i=1}^n v^1(x_i),
\end{equation}
where the one-body potential $v^1$ is defined by
\begin{equation}\label{substrate2}
v^1(x):=\sum_{s\in\mathcal{L}_S} v_\textrm{FS}(|x-s|)
\end{equation}
for any $x\in\Rz\times\{r\in\Rz\,:\, r>0\}$.  Notice that from the definition of $v_{FS}$ and $x_F$ 
$$v^1(x)\in\{\infty,0,-c_S, -2c_S\}$$
for any $x\in\mathcal{L}_F$.
%Note that the set of minimizers of $v_1$ form a curve contained in $\Rz\times[0,e_S]$ parametrizable as the graph of a periodic function from $\Rz$ to $[0,e_S]$ with period $1$.  Furthermore, if $0<e_{S}\ll1$, then $$v^1(x)\approx v_{S}(x^2)$$ for  any $x:=(x^1,x^2)\in\Rz^2$.  contained in $\Rz\times[0,e_S]$

Furthermore,  we considered the setting from the Introduction with \eqref{firstinterface_intro}, i.e.,
\begin{equation}\label{firstinterface}
{\rm dist}(\partial \mathcal{L}_F,\partial \mathcal{L}_S)\geq e_{FS},  %\RRR \quad  {\rm dist}(\mathcal{L}_F\setminus\partial\mathcal{L}_F,\partial\mathcal{L}_S)>e_{FS}, \BBB
\end{equation}
but without loss of generality we can actually reduce to the setting  
\begin{equation}\label{wlog_LFS}
\partial \mathcal{L}_{FS}\neq\emptyset
\end{equation}
 (which implies the equality in \eqref{firstinterface}) since otherwise by the choice of \eqref{substrate2}, 
the substrate interaction $E_S\equiv0$ and the same analysis of \cite{Yeung-et-al12}  applies, with the consequence that, up to rigid transformations, minimizers converge to a Wulff shape in $\Rz^2\setminus S$.  In particular, notice that   the value $-2c_S$ for the potential $v^1$ was always prevented in \cite{PiVe1}, while in our setting not. Moreover, we observe that we do not directly refer to the situation with  wall vectors in $Z^0_S$ and $q/p\in\Rz\setminus\Qz$, since in that case if $\partial \mathcal{L}_{FS}\neq\emptyset$, then  $\# \partial\mathcal{L}_{FS}\leq2$, and so the contribution of $E_S$ is also negligible and the analysis can be easily reduced to \cite{Yeung-et-al12} as well (as already noticed in \cite{PiVe1} for the case $x_F=x^0_F$ and $r=0$), with the consequence that, up to rigid transformations, minimizers converge to a Wulff shape in $\Rz^2\setminus S$ with the Wulff-shape boundary intersecting $\partial S$. 
Finally, by \eqref{wlog_LFS} we can always choose $x_F$ in the definition of $\mathcal{L}_F$ such that
\begin{equation}\label{wlog_x_F}
x_F \in \partial \mathcal{L}_{FS}.
\end{equation}

\subsection{Model comparison} \label{comparison} 

We observe that the setting with $x_F:=x_F^0$ with
\begin{equation}\label{eFSzero} x_F^0:=(0,e_{FS}),\end{equation} 
and $r=0$ in the definition of $\partial\mathcal{L}_S$ was already studied in \cite{PiVe1}, whose analysis we intend  here to generalize to the case $x_F\neq x_F^0$ with 
 %\RRR(to be probably changed depending what we do)\BBB
\begin{equation}\label{firstinterface2}
{\rm dist}(\partial \mathcal{L}_F,\partial \mathcal{L}_S)=e_{FS}.  %\RRR \quad  {\rm dist}(\mathcal{L}_F\setminus\partial\mathcal{L}_F,\partial\mathcal{L}_S)>e_{FS}, \BBB
\end{equation}
%which simply means that substrate atoms do not interact with second neighbors also through the film-substrate interface with respect to film atoms.
 The case of   $r\neq0$ is here introduced as an auxiliary setting since as shown in Section \ref{sec:main_results} we can reduce some other settings to the situation with $r\neq0$ and $x_F:=x_F^0$, which we throughly analyzed in Section \ref{sec:M1} for this reason obtaining analogous results of the ones in \cite{PiVe1}.

In order to perform such program, we need to compare the various settings. To this end we denote each discrete setting of Section \ref{sec:lattice_configurations} related to a specific positioning $x_F:=(x_F^1,x_F^2)\in\Rz^2\setminus\overline{S}$, wall vector $z:=(p,q,r)\in Z_S$, and vector
$$\Lambda:=(e_{FS},c_F,c_S)\in (\Rz^+)^3,$$  
which we refer to as the \emph{atomistic interaction vector}, where $\Rz^+:=\{s\in\Rz\,:\, s>0\}$, as the model 
\begin{equation}\label{models}
\text{$\mathcal{M}_{\Lambda}(x_F,z)$ or $\mathcal{M}_{\Lambda}(x_F,p,q,r)$},
\end{equation}
and  when comparing different models we also specify the dependence on $x_F$ and $z$ in the lattice notations, i.e., $\mathcal{L}_S:=\mathcal{L}_{S}(z)$, $\mathcal{L}_F:=\mathcal{L}_F(x_F)$ and so, also $\partial \mathcal{L}_{S}:=\partial \mathcal{L}_{S}(z)$, $\partial \mathcal{L}_{F}:=\partial \mathcal{L}_{F}(x_F)$, $\partial \mathcal{L}_{FS}:=\partial \mathcal{L}_{FS}(x_F,z)$, and on the interaction vector $\Lambda$ in the energies $V_{n}(D_n):=V_{n,\Lambda}(D_n)$ (for $v_{F\alpha}$ defined as in \eqref{defvf}  with respect to the parameters in $\Lambda$), for every configuration $D_n$ in the corresponding families $\mathcal{C}_n:=\mathcal{C}_n(x_F)$.  For simplicity such explicit dependence will be instead avoided in the notations when not comparing different models \eqref{models}.

Furthermore, for simplicity when $x_F=x_F^0$ we use the notation: %de single out some relevant models of the type \eqref{models} with a specific notation:
\begin{itemize}
\item[-] $\mathcal{M}_{\Lambda}^0(z):=\mathcal{M}_{\Lambda}(x_F^0,z)$ and with a slight abuse of notation, also $\mathcal{M}_{\Lambda}^0(p,q):=\mathcal{M}_{\Lambda}^0(z)$, when  $z:=(p,q,0)\in Z_S^0$ with
$$Z_S^0:=\{(p,q,r)\in Z_S : r=0\}=\{(p,q,0) : \text{$p$ and $q$ are co-prime integers}\};$$
\item[-] $\mathcal{M}_{\Lambda}^1(z):=\mathcal{M}_{\Lambda}(x_F^0,z)$ and with a slight abuse of notation, also $\mathcal{M}_{\Lambda}^1(q,r):=\mathcal{M}_{\Lambda}^1(z)$, when  $z:=(1,q,r)\in Z_S^1$ with 
$$Z_S^1:=Z_S\setminus Z_S^0=\{(p,q,r)\in Z_S : r\neq0\}=\left\{(1,q,r)\in\Nz^3 : 0<r\leq\frac{q}{2}\right\}.$$
%\item[-] $\mathcal{M}_{\Lambda}(x_F,z):=\mathcal{M}_{\Lambda}(x_F,z)$ when  $z\in Z_S^0$. 
\end{itemize} 
\RRR   \BBB
Notice that each model $\mathcal{M}_{\Lambda}^0(p,q)$  coincides with the model analyzed in \cite{PiVe1} and that $\mathcal{M}^1_{\Lambda}(2,1)=\mathcal{M}^0_{\Lambda}(1,1)$ for every $\Lambda\in(\Rz^+)^3$. % and that $\mathcal{M}_{\Lambda}(x_F^0)=\mathcal{M}_{\Lambda}^0$. Furthermore, in any model $\mathcal{M}_{\Lambda}(x_F)$ the distance between two atoms in $\partial \mathcal{L}_S(z)$ is constant and equal to
Furthermore, for every model $\mathcal{M}_{\Lambda}(x_F, z)$ with $z=(p,q,0)\in Z_S^0$ the distance between two atoms in $\partial \mathcal{L}_S(z)$ is constant and equal to 
\begin{equation}\label{eSratio}
e_S:=\frac{q}{p}
\end{equation}
which we refer to in these settings as the optimal \emph{substrate bond distance}.

Furthermore,  we  introduce an equivalence relation among models of type \eqref{models}, which turns out to be an \emph{equivalence relation} among  models of type \eqref{models}.

\begin{definition}\label{equivalent_model}
Let $\mathcal{M}^{\alpha}:=\mathcal{M}_{\Lambda^{\alpha}}(x_F^{\alpha},z^{\alpha})$ be two models with film-lattice centers $x_F^{\alpha}\in\Rz^2\setminus\overline{S}$, wall vectors $z^{\alpha}$, and interaction vectors $\Lambda^{\alpha}\in(\Rz^+)^3$ with respect to an index parameter $\alpha\in\{a,b\}$. Let also $\mathcal{C}_{n}^{\alpha}:=\mathcal{C}_{n}(x_F^{\alpha})$ and $V_n^{\alpha}:= V_{n,\Lambda^{\alpha}}$ be the related family of configurations and energy, respectively.  We say that
 the model $\mathcal{M}^a$ is equivalent to  $\mathcal{M}^b$ if 
 $$V_n^a(D_n^a)=V_{n}^b(D_n^b)$$ for every $D_n^a\in \mathcal{C}_n^a$, where $D_n^b:=D_n^a-x_F^a+x_F^b$ is referred to as the \emph{associated configuration} in $\mathcal{C}_n^b$ of the configuration $D_n^a\in\mathcal{C}_n^a$. 

\end{definition} 	

We recall that in \cite[Section 8]{PiVe1} the authors already addressed the existence of some examples of models of the type \eqref{models} which are equivalent to $\mathcal{M}^0_{\Lambda'}(z')$ for proper choices of the interaction vectors $\Lambda'\in (\Rz^+)^3$ and the wall vectors  $z'\in Z_S^0$ (see Figure \ref{fig:lattices}).

\begin{figure}
\frame{\includegraphics[width=0.482\textwidth]{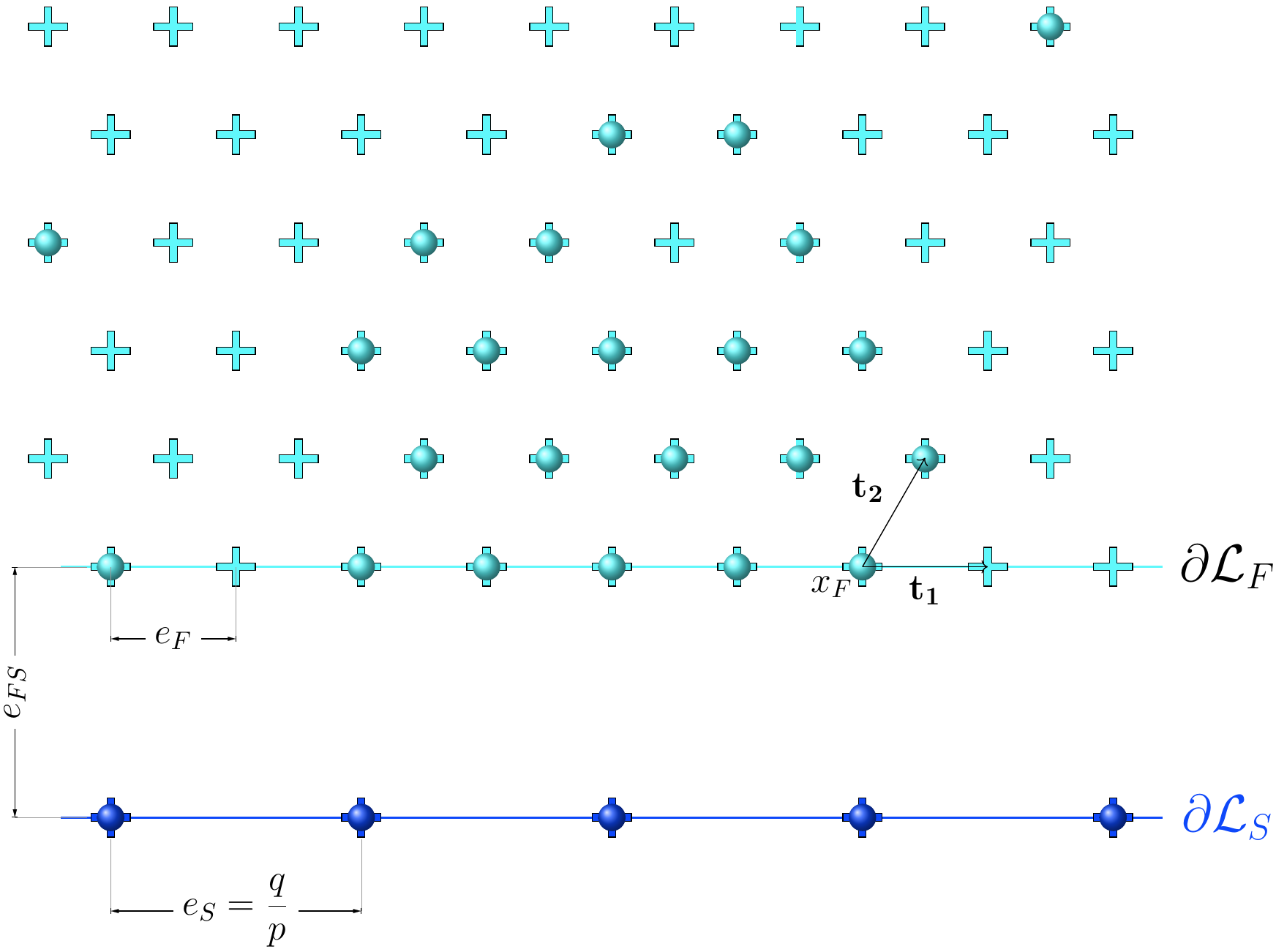}}\quad\frame{\includegraphics[width=0.48\textwidth]{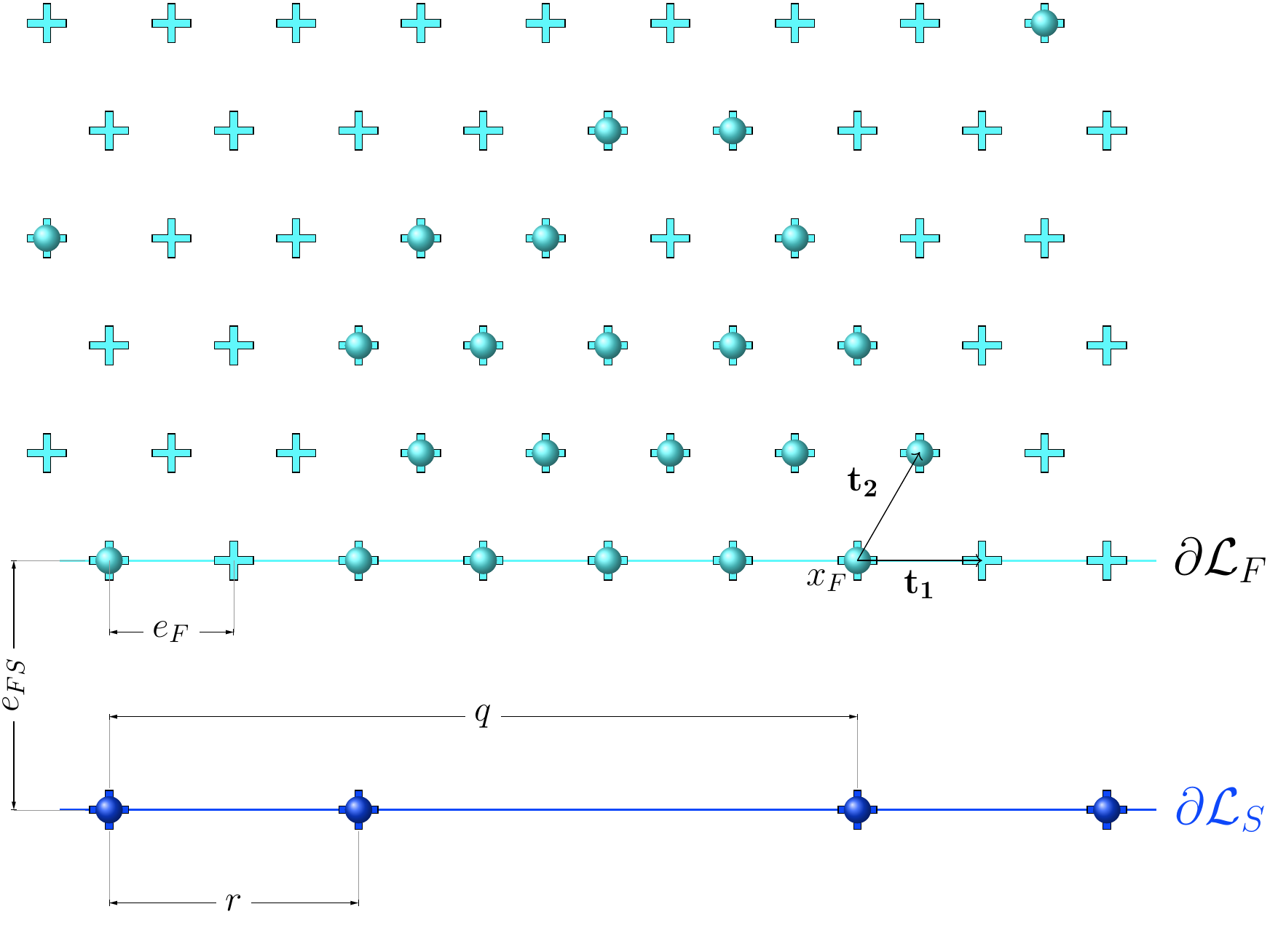}}
\caption{\PPP Two examples of models $\mathcal{M}_{\Lambda}(x_F,p,q,r)$ of the type \eqref{models} are depicted: on the left $\mathcal{M}_{\Lambda}^0(1,2)$ and on the right $\mathcal{M}_{\Lambda}^1(6,2)$. In both settings we can see a portion of the lattice border of the substrate and  a portion of the film lattice, i.e.,  $\partial\mathcal{L}_S$ and  $\mathcal{L}_F$, respectively. While  $\partial\mathcal{L}_S$ is fully occupied with substrate atoms represented by dark-blue balls, only the sites of $\mathcal{L}_F$ related to two  configurations, which are associated configurations with respect to Definition \ref{equivalent_model},  are occupied with film atoms represented by light-blue balls. 
}
\label{fig:lattices}
\end{figure}

\PPP %\subsection*{$\Gamma$-convergence result} 
\subsection{Setting with Radon measures} \label{radon_model} %As in \cite{Yeung-et-al12}, 
The $\Gamma$-convergence result is established  for a version of the previously described discrete model expressed in terms of \emph{empirical measures} since it is obtained with respect to the weak* topology of Radon measures \cite{AFP}. We denote the space of Radon measures on $\Rz^2$ by $\mathcal{M}(\Rz^2)$ and we write $\mu_{n}\stackrel{*}{\rightharpoonup} \mu$ to denote
the  convergence of  a sequence  $\{\mu_n\}\subset\mathcal{M}(\Rz^2)$ to a measure $\mu\in\mathcal{M}(\Rz^2)$ with respect to the weak* convergence of measures. 

Fix a model $\mathcal{M}_{\Lambda}(x_F,z)$ for a film-lattice center $x_F\in\Rz^2\setminus\overline{S}$, a wall vector $z\in Z_S$, and an interaction vector $\Lambda:=(e_{FS},c_F,c_S)\in (\Rz^+)^3$. The empirical measure $\mu_{D_n}$ \emph{associated to a configuration} $D_n:=\{x_1,\dots,x_n\}\in\mathcal{C}_n(x_F)$ is defined by 
 \begin{equation} \label{empiricalmeasures}
 \mu_{D_n}:=\frac{1}{n} \sum_{i=1}^n \delta_{\frac{x_i}{\sqrt n}}, 
\end{equation}
\FFF where $\delta_z$ represents the Dirac measure concentrated at a point $z\in \Rz^2$, \EEE and the family of empirical measures related to configurations in $\mathcal{C}_n(x_F)$ is denoted by $\mathcal{M}_n(x_F)$, i.e.,
\begin{equation}\label{radon_space}
\mathcal{M}_n(x_F):=\{\mu\in\mathcal{M}(\Rz^2)\ :\ \text{there exists $D_n\in\mathcal{C}_n(x_F)$ such that $\mu=\mu_{D_n}$}\}.
\end{equation}

The  functional $I_{n,\Lambda,x_F}$  associated to the configurational energy $V_{n,\Lambda}$ and expressed in terms of Radon measures is given by 
\begin{equation}
 I_{n,\Lambda,x_F}(\UUU\mu\EEE)\PPP:=
 \EEE\left\{\begin{array}{lll} % \int_{(\mathbf{R}^2_{+})^2 \backslash \textrm{diag})} n^2 v_{FF}(n^{1/2}|x-y|) d \mu_{D_n} \otimes d \mu_{D_n}, 
  \int_{(\PPP\mathbf{R}^2\setminus\overline{S}\EEE)^2 \backslash \textrm{diag}} n^2 v_{FF}(n^{1/2}|x-y|) d \mu\UUU(x)\EEE 
  \otimes d \mu\UUU(y) &%+ \int_{\PPP\mathbf{R}^2\setminus\overline{S}\EEE} n \UUU v^1\EEE (n^{1/2}x) d \UUU \mu(x) \EEE \EEE\quad\,\,\, 
  &\textrm{\PPP if $\mu\in\mathcal{M}_n(x_F)$, } %\UUU\mu\EEE=\frac{1}{n} \sum_{i=1}^n \delta_{x_i/\sqrt{n}}
  \\ 
 %\hspace{+40ex}  
&\hspace{-20ex} \RRR \EEE + \int_{\PPP\mathbf{R}^2\setminus\overline{S}\EEE} n \UUU v^1\EEE (n^{1/2}x) d \UUU \mu(x) \EEE &  \label{radon_functional}\\
%&\quad\textrm{for \PPP some $\{x_i\}\in\mathcal{C}_n$}\EEE, \\ 
& \\  +\infty &&\quad\textrm{otherwise,} 
\end{array}\right.
\end{equation}
\FFF where \EEE
$$ 
\textrm{diag}:=\{(y_1,y_2) \in \Rz^2: y_1=y_2\}. 
  $$	

%where $\mathbf{R}^2_{+}:=\Rz\times\{r\in\Rz\,:\, r>0\}$.
\PPP We notice that  the two versions of the discrete model are equivalent, since 
\begin{equation}\label{energy_equivalence}
V_{n,\Lambda}(D_n)=I_{n,\Lambda,x_F}(\mu_{D_n})
\end{equation}
for every  configuration $D_n\in\mathcal{C}_n(x_F)$, where $\mu_{D_n}\in\mathcal{M}_n(x_F)$ is defined by \eqref{empiricalmeasures}, and that $D_n$ minimizes $V_{n,\Lambda}$ among crystalline configurations in $\mathcal{C}_n(x_F)$ if and only if $\mu_{D_n}$ minimizes $I_{n,\Lambda,x_F}$  among Radon measures of $\mathcal{M}(\Rz^2)$.

\PPP \subsection{\PPP Local and strip energies} \label{local_energy}  

For the specific models $\mathcal{M}^0_{\Lambda}(z)$ and $\mathcal{M}^1_{\Lambda}(z)$ with $x_F:=x_F^0$ introduced in Section \ref{comparison} we consider localized energies  which together  contribute to the overall energy of configurations. To this end, fix a model $\mathcal{M}_{\Lambda}(x_F^0,z)$  for a choice of $z\in Z_S$ and $\Lambda:=(e_{FS},c_F,c_S)\in (\Rz^+)^3$. We  define the \emph{local energy} \PPP $E_{\rm loc}$  per site $x\in\mathcal{L}_F(x_F^0)$ with respect to a configuration $D_n$, \EEE by
 \begin{equation}\label{Eloc}
 E_{\rm loc}(x):=\begin{cases}
\sum_{y\in D_n\setminus\{x\}} v_{FF}(|x-y|) \,+\,6c_F 
 &\text{if $x\in D_n$,}\\
 0 &\text{if $x\notin D_n$},
 \end{cases}
 \end{equation}
\PPP which corresponds in the case of an atom $x\in D_n$ to the number of missing film bonds of $x$.
 We also refer to deficiency $ E_{\rm def}(x)$ of a site $x\in\mathcal{L}_F(x_F^0)$ with respect to a configuration $D_n$ as to the quantity
  \begin{equation}\label{deficiency}
 E_{\rm def}(x):=  
 \begin{cases}
 E_{\rm loc}(x)\,+\,v^1(x)
 &\text{if $x\in D_n$,}\\
 0 &\text{if $x \notin  D_n$}.
 \end{cases}
  \end{equation}
Furthermore, we define the \emph{strip} $\mathcal{S}(x)$ associated to any lattice site \PPP $x:=(x^1,e_{FS})\in D_n\cap\partial\mathcal{L}_{FS}$ with $x_1\in\Rz$ \EEE as the collection of atoms 
\begin{equation}\label{strip}
\mathcal{S}(x)=\mathcal{S}_{D_n}(x):=\{x,x_{\pm},\tilde{x},\tilde{x}_{\pm}\}\cap D_n
\end{equation}
where \PPP $x_{\pm}$, \EEE $\tilde{x}$,  and $\tilde{x}_{\pm}$ are defined by
\begin{align*}
&x_{\pm}:=x\pm\tone,\\
&\tilde{x}:=(x^1,y_M)\quad \text{where}\quad y_M:=\max\{y\geq0\,:\, (x^1,y)\in D_n\},\\
&\tilde{x}_{+}:=\tilde{x}+\ttwo,\\
&\tilde{x}_{-}:=\tilde{x}+\ttwo-\tone
\end{align*}
\PPP (see Figure \ref{fig:strip}). \EEE
\begin{figure}
\includegraphics[width=0.4\textwidth]{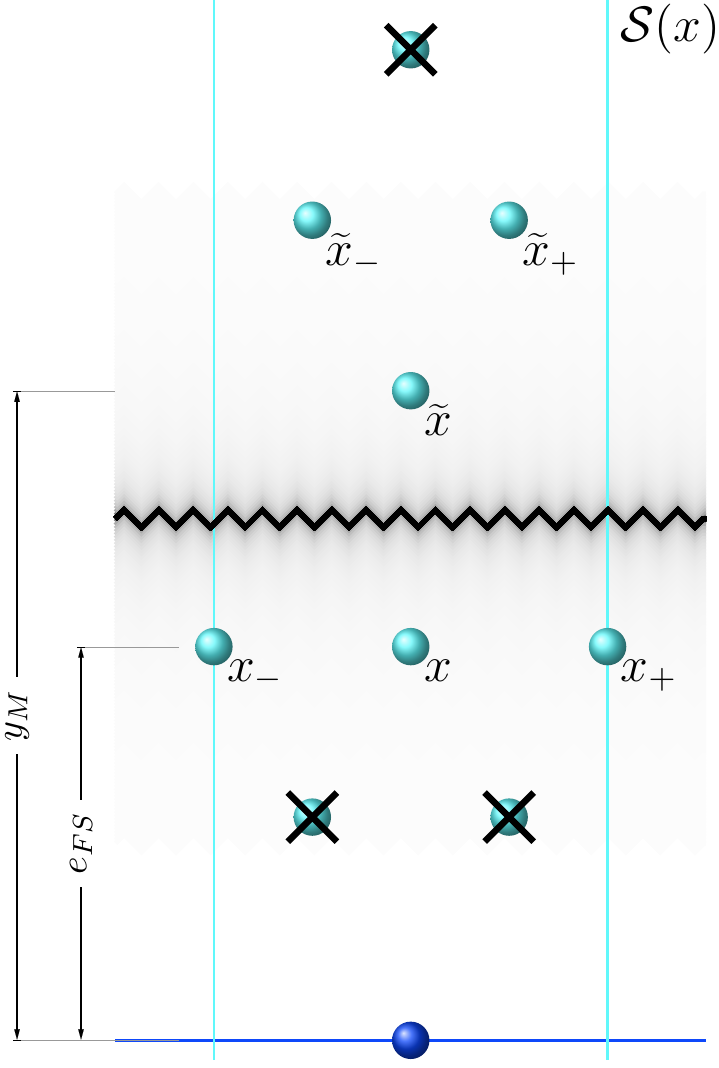}
\caption{\FFF For models of types $\mathcal{M}^0_{\Lambda}(z)$ and $\mathcal{M}^1_{\Lambda}(z)$ with $x_F:=x_F^0$ \EEE the strip $\mathcal{S}(x)$ centered at a atom $x\in\partial\mathcal{L}_{FS}$ of a crystalline configuration $D_n$ is depicted as an example of a strip containing all the elements $x,x_{\pm},\tilde{x},\tilde{x}_{\pm}$ \PPP  with the possibility of the \emph{strip center} $x$ and the \emph{strip top} $\tilde{x}$ to coincide if $y_M=e_{FS}$. The sites  indicated by crossed atoms are sites of the planar lattice $\{x_F+k_1 \tone+k_2 \ttwo\,:\, \textrm{$k_1, k_2\in\Zz$}\}$ that surely are not in $D_n$ by definition of $\mathcal{L}_F$ and  $\mathcal{S}(x)$.
}
\label{fig:strip} 
\end{figure}
In the following we refer to $x$ as the \emph{strip center} of $\mathcal{S}(x)$, to $x_{\pm}$ as \PPP the \EEE \emph{strip lower (right and left) sides}, to $\tilde{x}$ as  \PPP the \EEE  \emph{strip top}, and to $\tilde{x}_{\pm}$ as  \PPP the \EEE \emph{strip above (right and left) sides}.  Note \PPP that $x$ and $\tilde{x}$ coincide if $y_M=\SSS e_{FS}\BBB $.  

We define the \emph{strip energy} associated to a strip $\mathcal{S}(x)$ by 
\begin{equation}\label{Estrip}
  E_{\rm strip}(x):= E_{\rm strip, below}(x)\,+\,E_{\rm strip, above}({\color{black} x}),
\end{equation}
where
\begin{equation}\label{Estrip_below}\\
E_{\rm strip, below}(x):=  w(x) E_{\rm loc}(x)\,+\, w_{+}(x) E_{\rm loc}(x_{+})\,+\, w_{-}(x) E_{\rm loc}(x_{-})-c_S,
\end{equation}
with weights $w(x), w_{\pm}(x)\in\FFF \{1/4,1/2,1\}$ \EEE% {\color{pink}\{\frac{1}{2},1 \}}$ 
defined by
\begin{subnumcases}{\hspace{-1.2cm}(w(x), w_{+}(x),w_{-}(x)):=} 
  \left(1, \frac12,\frac12\right) & if either $z \in Z_S$ and $q\neq 1$,\notag\\
  & or $z \in Z_S^1$ with $z_2\neq1$ and $q=p=1$,\label{Eweight_below1}\\
     \left(\frac12, \frac12,\frac12\right) & if $z \in Z_S^1$ with $z_2=1$ and $q=p=1$, \label{Eweight_below3}\\\notag\\
  \left(\frac12, \frac14,\frac14\right)& if either $z \in Z_S^0 $ and $q=1$, \notag\\
  & or $z \in Z_S^1$ with $q=1$ and $p\neq1$, \label{Eweight_below2}

\end{subnumcases}

%6c_{F}-c_{S} & \text{if either $z \in Z_S$ and $q\neq 1$, or $z \in Z_S^1$ with $z_2\neq1$ and $q=p=1$},\\
%5c_{F}-c_{S}  & \text{if $z \in Z_S^1$ with $z_2=1$ and $q=p=1$}, \\
%4c_{F}-c_{S}  & \text{if either $z \in Z_S^0 $ and $q=1$, or $z \in Z_S^1$ with $q=1$ and $p\neq1$},

%\begin{subnumcases}{ E_{\rm strip, below}(x):=}
%  E_{\rm loc}(x)\,+\, \frac12 E_{\rm loc}(x_{+})\,+\, \frac12 E_{\rm loc}(x_{-})-c_S & if either $z\in Z_S^0$ and $q\neq1$, or $z\in Z_S^1$ and $z_2\neq1$ \label{Estrip_below1}
%   \\
%\frac12 E_{\rm loc}(x)\,+\, \frac14 E_{\rm loc}(x_{+})\,+\, \frac14 E_{\rm loc}(x_{-})-c_S & if $z\in Z_S^0$ and $q=1$ \label{Estrip_below2}\\
% \frac12 E_{\rm loc}(x)\,+\, \frac12 E_{\rm loc}(x_{+})\,+\, \frac12 E_{\rm loc}(x_{-})-c_S & if $z\in Z_S^1$ and $z_2=1$ \label{Estrip_below3}\\
%\end{subnumcases}
and    
\begin{equation}  \label{Estrip_above}
 E_{\rm strip, above}(x):=\begin{cases}
E_{\rm loc}(\tilde{x})\,+\, \SSS w_+\BBB(\tilde{x}) E_{\rm loc}(\tilde{x}_{+})\,+\, \SSS w_-\BBB(\tilde{x})  E_{\rm loc}(\tilde{x}_{-}) &\textrm{if $\tilde{x}\neq x$,}\\
\SSS w_+\EEE(\tilde{x}) E_{\rm loc}(\tilde{x}_{+})\,+\, \SSS w_-\BBB(\tilde{x}) E_{\rm loc}(\tilde{x}_{-}), &\textrm{if $\tilde{x}=x$}
  \end{cases}
\end{equation}
with weights $w_{\pm}(\tilde{x})\in\FFF \{1/2,1\}$ \EEE% {\color{pink}\{\frac{1}{2},1 \}}$ 
given by
\begin{equation}  \label{weigths_above}
w_{\pm}(\tilde{x}):=\begin{cases}
1&\textrm{if  $x_{\pm}\not\in\PPP D_n\cap\partial\mathcal{L}_{FS}$ or \EEE $\tilde{x}_{\pm}\neq \widetilde{(x_{\pm})}_{\mp}$,}\\
\frac12 &\textrm{if  $x_{\pm}\in\PPP D_n\cap\partial\mathcal{L}_{FS}$ and \EEE$\tilde{x}_{\pm}= \widetilde{(x_{\pm})}_{\mp}$}.
  \end{cases}
\end{equation}

  %%%%
  \subsection{\PPP Almost-connected configurations} \label{sec:almost_connected} \PPP 
  
  In this subsection we fix $x_F:=x_F^0$ and we introduce a weaker notion of  connectedness of configurations valid for the models $\mathcal{M}^0_{\Lambda}(z)$ and $\mathcal{M}^1_{\Lambda}(z)$, which is needed  to treat the situation when the wall parameter $q$ is not unitary, and so the energy is not invariant with respect to all horizontal translations of configurations. As we refer only to $\mathcal{M}^0_{\Lambda}(z)$ and $\mathcal{M}^1_{\Lambda}(z)$, we avoid in the section the dependence on the vectors $\Lambda\in (\Rz^+)^3$ and $z\in Z_S$. 
   
   We recall from Section  \ref{sec:lattice_configurations} that  a configuration $D_n$ is said to be connected  if every $x$ and $y$ in $D_n$ are connected through a path in the bonding graph of $D_n$, i.e., there exist $\ell\leq n$ and $x_k\in D_n$ for $k:=1,\dots,\ell$ such that $|x_k-x_{k-1}|=1$, $x_1=x$, and $x_\ell=y$, and we refer to maximal bonding subgraphs of $D_n$ connected through a path as \emph{connected components} of $D_n$. 

We say that a configuration $D_n$ is \emph{almost connected} if it is connected when $q=1$, and, if \FFF there exists an  enumeration of its $k:=k_{D_n}$ connected components, say   $D_n^i$, $i=1,\dots,k$, such that \FFF each \EEE $D_n^{i}$ is separated by at most $q$  from $\cup_{l=1}^{i-1}D_n^l$  for every $i=2,\dots,n$,  when $q\neq1$. 
  %{\color{pink} Notice that the definition depends on the \FFF enumeration \EEE of connected components; to eliminate such a dependence we actually demand that there exists a numeration of connected components such that the above property is valid.}
%  \ZZZ (Check: I put $pe_S$ here while there was only $e_S$. {\color{pink} This is fine} ) \PPP 
We say that a family of connected components of $D_n$ form an \emph{almost-connected component} of $D_n$ if their union is almost connected and, if $q\neq1$, it is distant from all other components of $D_n$ by more than $q$. % \ZZZ (Check: I put $pe_S$ here while there was only $e_S$ {\color{pink} This is fine} ). \EEE

\PPP
\begin{definition}\label{transformation} 
%We introduce a transformation of crystalline configuration that consists in performing translations first of the connected components and then of the almost-connected components of configurations. More precisely, given 
Given a configuration $D_n\in\mathcal{C}_n$, %we denote its $K_{D_n}\in\mathbb{N}$ connected components by $\mathcal{C}_i(D_n)$, $i=1,\dots,K_{D_n}$, and 
we define the \emph{transformed configuration} $\mathcal{T}(D_n)\in\mathcal{C}_n$ of $D_n$ as 
$$
\mathcal{T}(D_n):=\mathcal{T}_2(\mathcal{T}_1(D_n)),
$$ 
where $\mathcal{T}_1(D_n)$ is the configuration resulting by iterating the following procedure, starting from $D_n$:% $\mathcal{T}_1$ and $\mathcal{T}_2$ are the transformations defined by: 
 \begin{itemize}
 \item[-] If there are connected components without any activated bond with an atom of $\partial\mathcal{L}_S$, then select one of those components with lowest distance from $\partial\mathcal{L}_S$; 
 \item[-] Translate the component selected at the previous step of a vector in direction $-\ttwo$ till either a bond with another connected component or with the substrate is activated. %{\color{pink} In the case when the component has non-empty intersection with $\partial \mathcal{L}_F$, but still does not have activated bond with $S$, then we translate it in the direction $\tone$ in a way that it has at least one activated bond with $S$};
 \end{itemize}
 (notice that the procedure ends when all connected components of $\mathcal{T}_1(D_n)$ have at least a bond with  $\partial\mathcal{L}_S$), and  $\mathcal{T}_2(\mathcal{T}_1(D_n))$ is the configuration resulting by iterating the following procedure, starting from  $\mathcal{T}_1(D_n)$: 
  \begin{itemize}
   \item[-] If there are more than one almost-connected component, then select the almost-connected component whose  leftmost bond with  $\partial\mathcal{L}_S$ is the second (when compared with \FFF the \EEE other almost-connected components) starting from the left;
  % If there are more than one almost-connected component, then  consider the leftmost bond with  $\partial\mathcal{L}_S$ of each almost-connected component and select the almost-connected component whose  leftmost bond is the second starting from the left; 
       \item[-] Translate the  almost-connected component selected at the previous step of a vector  $-kq\tone$ for some $k \in \mathbf{N}$ till, if $q=1$,  a bond with another connected component is activated, or, if $q\neq1$, the distance with another almost-connected component is  less or equal to $q$; %\ZZZ (Check: I put equal to $pe_S$ and not ``less or equal'' with same energy as before, for simplicity.{\color{pink} This is fine.}) \EEE

% \item[-] Select the almost-connected component $a(\mathcal{T}_1(D_n))$ of $\mathcal{T}_1(D_n)$ with leftmost bond with  $\partial\mathcal{L}_S$;
% \item[-] Determine the leftmost atom $s\in\partial\mathcal{L}_S$ bonded of $\mathcal{T}_1(D_n)\setminus a(\mathcal{T}_1(D_n))$ and denote its almost-connected component by $a'(\mathcal{T}_1(D_n))$;
%  \item[-] Translate $a'(\mathcal{T}_1(D_n))$ with respect to the vector $-\tone$ till,  if $q=1$, a bond with another connected component  is activated, or, if $q\neq1$, the distance with another almost-connected component is equal to $pe_S$;
  \end{itemize} 
 (notice that the procedure ends when $\mathcal{T}_2(D_n)$ is almost connected). 
\end{definition}

\noindent We notice that the transformed configuration $\mathcal{T}(D_n)$ of a configuration $D_n\in\mathcal{C}_n$ satisfies the following properties: 
 \begin{itemize}
 \item[(i)] $\mathcal{T}(D_n)$ is almost connected;
  \item[(ii)]  Each  connected component of $\mathcal{T}(D_n)$ includes at least an atom bonded to $\partial\mathcal{L}_S$;
  \item[(iii)] $V_n(\mathcal{T}(D_n))\leq V_n(D_n)$ (as no active bond of $D_n$ is deactivated by performing the transformations $\mathcal{T}_1$ and $\mathcal{T}_2$);
  \end{itemize}
  and,  if $D_n$ is a minimizer of $V_n$ in $\mathcal{C}_n$, then 
   \begin{itemize}
 \item[(iv)] $\mathcal{T}_1(D_n)\SSS = \BBB D_n$;
  \item[(v)] $\mathcal{T}$ consists of translations of the almost-connected components of $D_n$ with respect to a vector (depending on the component)  in the direction $-\tone$ with norm  in $\Nz\cup\{0\}$. 
  \end{itemize}
  
  \FFF Finally we also observe that the definitions of $\mathcal{T}_1$, $\mathcal{T}_2$, and $\mathcal{T}$ are independent from $n$. 

\EEE

  \subsection{\PPP Continuum setting} \label{sec:continuum_model} \EEE
%Fix $x_F\in\Rz^2\setminus\overline{S}$, $z\in Z_S$ and $\Lambda:=(e_{FS},c_F,c_S)\in (\Rz^+)^3$ and let
%$$\Sigma:=(x_F,z,\Lambda)\in \Rz^2\setminus\overline{S}\times Z_S\times(\Rz^+)^3.$$
We define  the \emph{anisotropic surface tension} $\Gamma:\mathbb{S}^{1} \to \mathbb{R}$ as the function for which
\begin{equation} \label{gamma}
\Gamma (\nu(\varphi))=2 c_F\left(\nu_2(\varphi)-\frac{\nu_1(\varphi)}{\sqrt{3}}\right)
  \end{equation} 
 for every
$$
\nu(\varphi)=\left(\begin{array}{c} -\sin \varphi \\ \cos \varphi   \end{array} \right)\in\mathbb{S}^{1}\quad\text{with}\quad  \varphi \in \left[0, \frac{\pi}{3}\right],
$$
such that  $\Gamma\circ\nu$ is extended \EEE periodically on $\Rz$ as a  $\pi/3$-periodic function. Notice  that  $\SSS\Gamma (\pm\tthree)=2c_F\BBB$ and that by extending $\Gamma$ by homogeneity we obtain a convex function, and in particular  a  Finsler norm on $\Rz^2$. \EEE
%The following condition will be referred to as the \emph{Wulff condition), %\emph{Wulff regime condition):
%\begin{equation}\label{wulff_condition} 
%\begin{cases}

%\end{cases}
%\end{equation}

\PPP For every set of finite perimeter $D\subset\Rz^2\setminus S$ we formally define its anisotropic surface energy $\mathcal{E}$ by \EEE
\begin{equation} \label{continuum}
\mathcal{E}_{\sigma}(D):= \int_{\SSS \partial^* D \backslash \partial S \BBB} \Gamma (\nu_D) d \mathcal{H}^1+\sigma\mathcal{H}^1(\partial^*D \cap \partial S)
%\mathcal{E} (D):= \int_{\SSS \partial^* D \backslash \partial S \BBB} \Gamma (\nu_D) d \mathcal{H}^1+\left( 2c_F-\frac{c_S}{q}\right) \mathcal{H}^1(\partial^*D \cap \partial S)
\end{equation}
\PPP where $\partial^* D$ denotes the reduced boundary of $D$ and $\sigma\in\Rz$ is a parameter representing the adhesivity. Notice that in passing from the discrete to the continuum setting  $\sigma$ will be characterized in terms of the parameters of the discrete settings with conditions in particular entailing the lower semicontinuity of $\mathcal{E}_{\sigma}$.

%\PPP We also use the following \emph{auxiliary surface energy} \FFF depending on $n$ \EEE in the proofs 
%\begin{equation} \label{continuum_n}
%\mathcal{E}_{n,\sigma} (D):= \int_{\partial^* D \cap (\Rz^2\setminus \overline{S_n})} \Gamma (\nu_D) d \mathcal{H}^1+\sigma \mathcal{H}^1(\partial^*D \cap \partial S_n) \end{equation}
%where  
%\begin{equation}\label{substrate_n}
%S_n:=S+\frac{\SSS e_{FS} \BBB}{\sqrt n}\tthree.
% \end{equation}

  \subsection{\PPP Main results} \label{sec: main_result} \EEE
% The first result is a crystallization result consisting in singling out a condition depending only on the constants $c_F$ and $c_S$ characterizing the wetting regime, i.e., the situation in which all the minimizers of $V_n$ are configurations contained in $\partial\mathcal{L}_F$. %the minimizers of which the aim of the paper is $\Gamma$-limit of $n^{1/2}(I_n + 6c_Fn)$ with respect to the weak* topology is $\mathcal{E}$.  First we will prove the compactness result, which justifies the used topology, and then we prove lower and upper bound. 

 In this section the main theorems proven in the manuscript are stated.  Let us classify all the models  $\mathcal{M}_{\Lambda}(x_F,z)$ of type \eqref{models} for 
interaction vectors $\Lambda:=(e_{FS},c_F,c_S)\in (\Rz^+)^3$, film-lattice centers $x_F\in\Rz^2\setminus\overline{S}$,  and wall vectors $z\in Z_S^0$ such that \eqref{firstinterface} holds (see Figure \ref{fig:mainmodel}), in four categories, namely  by saying that $\mathcal{M}_{\Lambda}(x_F,z)\in C_i$ with the classes of models $C_i:=C_i(\Lambda,x_F,z)$ given for $i=1,\dots,4$ by: 
\begin{figure}
\includegraphics[width=0.68\textwidth]{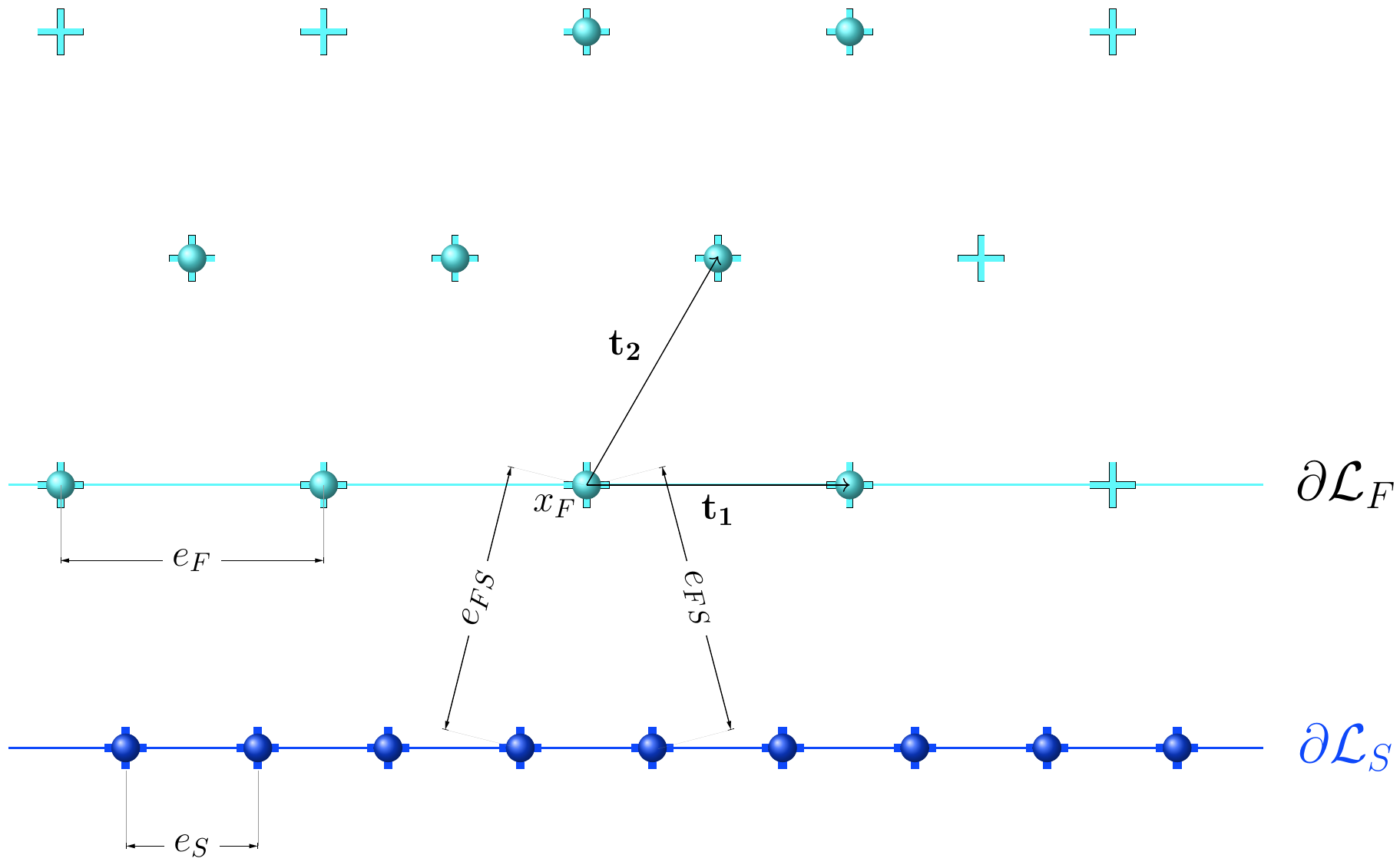}
\caption{\PPP Example of a model  $\mathcal{M}_{\Lambda}(x_F,z)$ of type \eqref{models} with %film-lattice center $x_F:=(-e_S/2,\sqrt{e_{FS}^2-e_S^2/4})$ and 
wall vectors $z=(2,1,0)\in Z_S^0$, which is in particular in the class $C_1$.}
\label{fig:mainmodel}
\end{figure}

\begin{itemize}
%\item $C_1:=\{\mathcal{M}_{\Lambda}(x_F,z) : \text{$x_F$ is bonded with two substrate atoms}\}$;
\item[$C_1$:] models with $x_F$ having two substrate neighbors;
\item[$C_2$:] models with $x_F$ having exactly one substrate neighbor and 
$$(x_F,x_F+q \tone)\cap \partial \mathcal{L}_{FS}(x_F,z)=\emptyset;$$
\item[$C_3$:] models with $x_F$ having exactly one substrate neighbor and 
 $$(x_F, x_F+q\tone)\cap \partial \mathcal{L}_{FS}(x_F,z)=\{x_F+\frac{q}{2} \tone\}$$
 (which  implies that $q$ is even);
\item[$C_4$:] model with $x_F$ having exactly one substrate neighbor,  $q>2$, and for which there exists $s \in \{1,\dots, q-1\}\setminus \{q/2\}$ such that 
$$(x_F, x_F+q\tone)\cap \partial \mathcal{L}_{FS}(x_F, z) =\{x_F+s\tone\}.$$
%In this case we have 
%$$\partial \mathcal{L}_{FS}(z_0)= \{x_F+k: k \in q\mathbb{Z}\cup (q\mathbb{Z}+r) \}.         $$ 
  %$(x_F, x_F+q\tone)\cap \partial \mathcal{L}_{FS}(z)\neq \emptyset$ and $(x_F, x_F+q\tone)\cap \partial \mathcal{L}_{FS}(z)\neq \{x_F+\frac{q}{2} \tone\}$.
\end{itemize}

Notice that by the proof of Proposition  \ref{equivalence_Z1} in particular follows that for every  model $\mathcal{M}_{\Lambda}(x_F,z)$ with $z\in Z_S^0$ there exists a unique $i_0\in\{1,\dots,4\}$ such that $\mathcal{M}_{\Lambda}(x_F,z)\in C_{i_0}$, and that   any model $\mathcal{M}_{\Lambda}(x_F,z)\in C_i$ for $i\in\{1,\dots,4\}$ is such  that  $\partial \mathcal{L}_{FS}:=\{x_F+k : k\in L_i\}$ with 
$$ 
L_i:=\begin{cases}
q\Zz &\text{if $i=1,2$,}\\
\frac{q}{2}\Zz &\text{if $i=3$,}\\
q\Zz\cup (q\Zz+s) &\text{if $i=4$.}\\
\end{cases}
$$

%$\mathcal{L}_S:=\mathcal{L}_Sz$, $\mathcal{L}_F:=\mathcal{L}_F(x_F)$, and $\mathcal{C}_n:=\mathcal{C}_n(x_F)$, and $V_{n,\Lambda}(D_n)$ (for $v_{F\alpha}$ defined as in \eqref{defvf}  with respect to the parameters in $\Lambda$).

\PPP We begin with \SSS the \BBB following result that characterizes  the wetting regime in terms of a condition only depending on $v_{FF}$ and $v_{FS}$, and  the  minimizers in such regime, which we denote as \emph{wetting configurations}. More precisely, we say that a configuration $D^{\rm w}_n:=\{w_1,\dots,w_n\}\in\mathcal{C}_n(x_F)$ is a \emph{wetting configuration} 
\begin{itemize}
\item[-] if $D^{\rm w}_n\subset\partial\mathcal{L}_{FS}(x_F,z)$ when either $\mathcal{M}_{\Lambda}(x_F,z)\in C_i$ for $i=1,2$ and $q>1$, or $\mathcal{M}_{\Lambda}(x_F,z)\in C_3$ and $q>2$, or $\mathcal{M}_{\Lambda}(x_F,z)\in C_4$ and $1<s<q-1$;
%for either $z \in Z_S$ and $q\neq 1$, or $z \in Z_S^1$ with $z_2\neq1$ and $q=p=1$ if $D^{\rm w}_n\subset\partial\mathcal{L}_{FS}(x_F,z_1,z_2)$,
\item[-] if $D^{\rm w}_n\subset\partial\mathcal{L}_{FS}(x_F,z)$ and  
\begin{equation}\label{connected_wetting}
w_{i+1}:=w_i+\tone
\end{equation}
for every $i=1,\dots,{n}$, when either $\mathcal{M}_{\Lambda}(x_F,z)\in C_i$ for $i=1,2$ and $q=1$ or $\mathcal{M}_{\Lambda}(x_F,z)\in C_3$ and $q=2$;
\item[-] if $D^{\rm w}_n\subset\partial\mathcal{L}_{FS}(x_F,z)$ and for every $i=1,\dots,{n}$ up to one index in the case $n$ is odd we have that either $w_i+\tone$ or $w_i-\tone$ belongs to $D^{\rm w}_n$ as well, when $\mathcal{M}_{\Lambda}(x_F,z)\in C_4$ and $s=1$ or $s=q-1$. 
\end{itemize}

 \PPP
 \begin{theorem}[Wetting regime]\label{wetting_theorem}
For every $n\in\Nz$ any  wetting configuration $D^{\rm w}_n\in\mathcal{C}_n(x_F)$ satisfies \EEE the following two assertions: 
\begin{itemize} 
	\item[(i)] $V_n(D^{\rm w}_n)=\min{V_n(D_n)}$,
	\item[(ii)] \PPP $V_n(D^{\rm w}_n)<V_n(D_n)$ \EEE for  \PPP every crystalline configuration $D_n\in\mathcal{C}_n(x_F)$ that is not a wetting configuration,
%\item[(ii)] $\PPP V_n(D^{\rm w}_n) <  V_n(D_n)$ \EEE for  \PPP any crystalline \EEE configuration $D_n$ with $D_n\cap\mathbb{R}\times\{r>e_S \}\neq\emptyset$\PPP, \EEE \ZZZ (we do not have strict inequality?)\EEE
\end{itemize}
if and only if 
%c_{S}\geq 6 c_{F} if $z \in Z_S^1$ with either $q\neq1$, or $q=p=1$ and $z_2\neq1$
%c_{S}\geq 5 c_{F} if $z \in Z_S^1$ with $z_2=1$ and $q=p=1$
%c_{S}\geq 4 c_{F} if $z \in Z_S^1$ with $q=1$ and $p\neq1$
\begin{equation}\label{wetting_condition_total} 
\hspace{-0.2cm}\begin{cases}
c_{S}\geq 6c_{F} & \text{if either $\mathcal{M}_{\Lambda}(x_F,z)\in C_i$ for $i=1,2$ and $q>1$, or $\mathcal{M}_{\Lambda}(x_F,z)\in C_3$}\\ 
& \text{and $q>2$, or $\mathcal{M}_{\Lambda}(x_F,z)\in C_4$ and $1<s<q-1$},\\
c_{S}\geq 5c_{F} & \text{if $\mathcal{M}_{\Lambda}(x_F,z)\in C_4$ and $s=1$ or $s=q-1$}, \\
c_{S}\geq 4c_{F} & \text{if either $\mathcal{M}_{\Lambda}(x_F,z)\in C_i$ for $i=1,2$ and $q=1$, or $\mathcal{M}_{\Lambda}(x_F,z)\in C_3$}\\  
& \text{and $q=2$}. 
\end{cases}
\end{equation}
%\begin{equation}\label{wetting_condition_total} 
%\begin{cases}
%c_{S}\geq 6c_{F} & \text{if either $z \in Z_S^0$ and $q\neq 1$, or $z \in Z_S^1$ and $z_2\neq 1$},\\
%c_{S}\geq 5c_{F} & \text{if $z\in Z_S^1$ and $z_2=1$}, \\
%c_{S}\geq 4c_{F} & \text{if $z \in Z_S^0 $ and $q=1$}. 
%\end{cases}
%\end{equation}
\FFF In particular, for the necessity of \eqref{wetting_condition_total} it is enough assertion (i), and more specifically that 
%or more precisely it is enough that 
%. Notice that such necessity 
%The reverse implication in the Theorem \ref{wetting_theorem} 
%should be understood  in the following way: If 
there exists an increasing subsequence $(n_k)_{k \in \mathbf{N}}$ such that $V_{n_k}(D^{\rm w}_{n_k})=\min{V_{n_k}(D_{n_k})}$ holds for every $n_k$. \EEE	
\end{theorem} 

%\FFF \begin{remark} For the necessity of \eqref{wetting_condition_total} it is enough assertion (i), or more precisely it is enough that 
%. Notice that such necessity 
%The reverse implication in the Theorem \ref{wetting_theorem} 
%should be understood  in the following way: If 
%there exists an increasing subsequence $(n_k)_{k \in \mathbf{N}}$ such that (i) holds for every $n_k$. 		
%\end{remark} \EEE	
We refer to  \eqref{wetting_condition_total} as a \emph{wetting condition} or as the \emph{wetting regime}, and to the opposite condition, namely 
\begin{equation}\label{dewetting_condition} 
\begin{cases}
c_{S}< 6c_{F} & \text{if either $\mathcal{M}_{\Lambda}(x_F,z)\in C_i$ for $i=1,2$ and $q>1$, or $\mathcal{M}_{\Lambda}(x_F,z)\in C_3$}\\ 
& \text{and $q>2$, or $\mathcal{M}_{\Lambda}(x_F,z)\in C_4$ and $1<s<q-1$},\\
c_{S}< 5c_{F} & \text{if $\mathcal{M}_{\Lambda}(x_F,z)\in C_4$ and $s=1$ or $s=q-1$}, \\
c_{S}< 4c_{F} & \text{if either $\mathcal{M}_{\Lambda}(x_F,z)\in C_i$ for $i=1,2$ and $q=1$, or $\mathcal{M}_{\Lambda}(x_F,z)\in C_3$}\\  
& \text{and $q=2$,}
\end{cases}
\end{equation}
as the \emph{dewetting condition} or the \emph{dewetting regime}. \EEE %since \eqref{dewetting_condition}  contradicts both   \eqref{wetting_condition1} and \eqref{wetting_condition2} 
\PPP The following result shows that connected components with the largest cardinality of minimizers incorporate the whole mass in the limit. \EEE

\begin{theorem}[Mass conservation] \label{connectness} 
\PPP Assume \eqref{dewetting_condition}. If $ \widehat{D}_n$ are  minimizers of $V_{n,\Lambda}$ among all crystalline configurations in $\mathcal{C}_n$, i.e.,  
$$V_{n,\Lambda}(\widehat{D}_n)=\min_{D_n\in\mathcal{C}_n}{V_{n,\Lambda}(D_n)},$$
and we select for every $\widehat{D}_n$ a connected component $\widehat{D}_{n,1}\subset \widehat{D}_n$ with  largest cardinality, then
$$
\lim_{n \to \infty}{\mu_{\widehat{D}_n} (\widehat{D}_{n,1})}=1, 
$$
 %\ZZZ(I am a bit surprise that we can write this in general and not for specific translations of ${\color{pink}\widehat{D}_{n,1}}$, but checking the proof it seems right) \PPP 
where $\mu_{\widehat{D}_n}$ are the empirical measure associated to $\widehat{D}_n$ defined by \eqref{empiricalmeasures}. \EEE
\end{theorem}	

\PPP We rigorously prove by $\Gamma$-convergence that the discrete models converge to the continuum model, and in view of the previous result (even in the lack of a direct compactness result for general sequences of minimizers, possibly not almost connected), we prove convergence (up to passing to a subsequence and up to translations) of the minimizers of the discrete models to a bounded minimizer of the continuum model, which in \SSS turn \BBB  is also proven to exist. %\CCC We do not \CCC  discuss here further the minimality property of the Winterbottom shape for the energy $\mathcal{E}$ and the uniqueness of the minimizers of $I_{\infty}$ in  $\mathcal{M}_W$. \BBB

\begin{theorem}[Convergence of Minimizers]\label{thm:convergence_minimizers}
	\PPP Assume \eqref{dewetting_condition}. The following statements hold:
	\begin{itemize}
	\item[1.] The functional 
	\begin{equation}\label{converging_energies00}
	E_{n,\Lambda,x_F}:= n^{-1/2}(I_{n,\Lambda,x_F}+6c_F n),
	\end{equation}
	where $I_{n,\Lambda,x_F}$ is defined by \eqref{radon_functional}, $\Gamma$-converges with respect to the weak* convergence of measures to the functional $I_{\infty,i}$ defined by
		\begin{equation}
	\hspace{0.5cm}I_{\infty,i}(\mu):=\begin{cases}  \mathcal{E}_{\sigma_i}(D_\mu), &  \text{if \FFF there exists \EEE $D_\mu\subset\Rz^2\setminus S$  set of finite perimeter}\\
	&\hspace{0ex} \text{with  $|D_\mu|=1/\rho$ such that $\mu=\rho\chi_{D_\mu}$,}
\\%\text{if $\mu\in\mathcal{M}_W$},\UUU\\
	+\infty, &\text{otherwise,}\\
	\end{cases}
	\end{equation}
	for every $\mu\in\mathcal{M}(\Rz^2)$,
where $\rho:=2/\sqrt{3}$ and $\mathcal{E}_{\sigma_i}$ is defined in \eqref{continuum} for an adhesivity $\sigma:=\sigma_i$ given for $i=1,\dots,4$ by
\begin{equation}\label{sigma}
\sigma_{i}:=\begin{cases}
2c_F-\displaystyle\frac{2c_S}{q} & \text{if $i=1,3,4$ and $\mathcal{M}_{\Lambda}(x_F,z)\in C_i$,}\\ 
&\\
2c_F-\displaystyle\frac{c_S}{q}  & \text{if $i=2$ and $\mathcal{M}_{\Lambda}(x_F,z)\in C_2$.}\\
   \end{cases}
\end{equation}
%\end{subnumcases}
	 %, where 
%\begin{align*}
%	\mathcal{M}_W:=\bigg\{\mu\in\mathcal{M}(\Rz^2)\ :\ \text{$\exists$ $D\subset\Rz^2\setminus S$  set of finite perimeter, bounded, and with } & |D|=\frac{\sqrt{3}}{2}\\
%	& \hspace{-15ex} \text{such that $\mu=\frac{2}{\sqrt{3}}\chi_D$}\bigg\}.
%	\end{align*}
	\item[2.]  The functional $I_{\infty,i}$ admits a minimizer in 
	\begin{align}
\mathcal{M}_W:=\bigg\{\mu\in\mathcal{M}(\Rz^2)\ :\ \text{$\exists$ $D\subset\Rz^2\setminus S$} & \hspace{2ex} \text{set of finite  perimeter, bounded }\notag\\
\hspace{8ex} \text{with }  |D|=\frac{1}{\rho},\label{M_Wulff}	& \hspace{2ex} \text{ and such that $\mu=\rho\chi_D$}\bigg\}. 
	\end{align}
	\item[3.]  Every sequence $\mu_n\in\mathcal{M}_n(x_F)$ of minimizers of $E_{n,\Lambda,x_F}$ %of crystalline configurations $D_n^m\subset\mathcal{L}_F$  of the functional \UUU
%	\begin{equation}\label{converging_energies}
%	E_n:= n^{-1/2}(I_n+6c_F n)
%	\end{equation}
admits, up to translation %in the direction $\tone$ $($i.e., up to replacing $\mu_n$ with $\mu_n(\cdot+c_n\tone)$ for  chosen fixed integers $c_n\in\Zz$$)$, 
a subsequence  converging  with respect to the weak* convergence of measures to a minimizer of $I_{\infty,i}$ in $\mathcal{M}_W$. 
	\end{itemize}
	\end{theorem}

\PPP
\noindent Notice that the parameter $\rho:=2/\sqrt{3}$ in the definition of $\mathcal{M}_W$ is related to the fact that we chose the triangular lattice for $\mathcal{L}_F(x_F)$, as $\rho$ is the density of atoms per unit volume of such lattice. 

Finally, we observe that as a byproduct of our strategy analogous results as Theorems \ref{wetting_theorem}, \ref{connectness}, and \ref{thm:convergence_minimizers} are also obtained for any models $\mathcal{M}^1_{\Lambda}(z)$ with $z\in Z^1_S$ and $\Lambda\in (\Rz^+)^3$, namely in Theorems \ref{wetting_theorem_Z1}, \ref{connectness_Z1}, and \ref{thm:convergence_minimizers_Z1}, respectively.

%by can endow the class $X_n$ of configurations with $n$ particles with such equivalence relation $\mathcal{R}$ and
 \EEE

\section{\PPP Model $\mathcal{M}^1_{\Lambda}(z)$}\label{sec:M1}
In this section we prove analogous results of the ones contained in Section \ref{sec: main_result} for the specific model $\mathcal{M}^1_{\Lambda}(z)$  defined in Section \ref{comparison} for $z\in Z_S^1$ and $\Lambda:=(e_{FS},c_F,c_S)\in (\Rz^+)^3$, i.e., Theorems \ref{wetting_theorem_Z1}, \ref{connectness_Z1}, and \ref{thm:convergence_minimizers_Z1}, respectively. In this section for simplicity we often avoid indicating the dependence on  $x_F:=x_F^0$ and on $\Lambda\in (\Rz^+)^3$ as we only deal with models $\mathcal{M}^1_{\Lambda}(z)$.

We start  by characterizing the wetting regime for the model $\mathcal{M}^1_{\Lambda}(z)$. In the following we say that a configuration $D^{\rm w}_n:=\{w_1,\dots,w_n\}\in\mathcal{C}_n$ is a \emph{wetting configuration for} $\mathcal{M}^1_{\Lambda}(z)$ if:
\begin{itemize}
\item[-]  $D^{\rm w}_n\subset\partial\mathcal{L}_{FS}(z)$  for $z \in Z_S^1$ with $r>1$ (and so $q>2$);

%for either $z \in Z_S$ and $q\neq 1$, or $z \in Z_S^1$ with $z_2\neq1$ and $q=p=1$ if $D^{\rm w}_n\subset\partial\mathcal{L}_{FS}(x_F,z_1,z_2)$,
\item[-] $D^{\rm w}_n:=\{w_1,\dots,w_n\}\subset\partial\mathcal{L}_{FS}(z)$ with 
\begin{equation}\label{connected_wetting1}
w_{i+1}:=w_i+\tone
\end{equation}
for every $i=1,\dots,{n}$, for $z \in Z_S^1$ with $r=1$ and $q=2$;
\item[-] $D^{\rm w}_n:=\{w_1,\dots,w_n\}\in\mathcal{L}_{FS}(z)$ such that for every $i=1,\dots,{n}$ up to one index in the case $n$ is odd we have that either $w_i+\tone$ or $w_i-\tone$ belongs to $D^{\rm w}_n$ as well, for $z \in Z_S^1$ with $r=1$ and $q>2$. 
\end{itemize}

 \begin{proposition}\label{wetting_regime1}
Let $n \in \mathbb{N}$ and $z \in Z_S^1$ with $r>1$  \emph{(}and so $q>2$\emph{)}. %either $q\neq1$, or $q=p=1$ and $z_2\neq1$. 
Any wetting configuration $D^{\rm w}_n\subset\partial\mathcal{L}_{FS}(z)$ \EEE %such that 
%\begin{equation}\label{connected_wetting}
%w_{i+1}:=w_i+\tone
%\end{equation}
% for $i=1,\dots,n$ 
satisfies \PPP the following two assertions: \EEE
\begin{itemize} 
\item[(i)] $V_n(D^{\rm w}_n)=\min{V_n(D_n)}$\PPP, \EEE
\item[(ii)] \PPP $V_n(D^{\rm w}_n)<V_n(D_n)$ \EEE for  \PPP any $D_n$ that is not a wetting configuration,  %\ZZZ (I added strict inequality among all other configurations)\EEE
\end{itemize}
 if and only if 
\begin{equation}\label{wetting_condition1}
c_{S} \geq 6 c_{F}.
%\begin{cases}
%c_{S}>6 c_{F} & \textrm{if $r>1$},\\
%c_{S}>5 c_{F} & \textrm{if $r\leq1$}.
%\end{cases}
\end{equation}
  \end{proposition}

\begin{proof}

The proof is the same as the proof of Proposition 3.1 in \cite{PiVe1} for the model $\mathcal{M}^0$ with $q\neq1$, by observing that 
\begin{equation}\label{wetting_energy}
V_n (D^{\rm w}_n)=-c_S n.
\end{equation}
also with respect to the model $\mathcal{M}^1_{\Lambda}(z)$ for the conditions of the assertion on the wall parameters. % $z \in Z_S^1$ in the  case of either $q\neq1$, or $q=1$ and $z_2\neq 1$. 

\end{proof}

We now consider the case of $z \in Z_S^1$ with $r=1$ and $q=2$ (for which  $\mathcal{M}^1_{\Lambda}(q,r)=\mathcal{M}^0_{\Lambda}(1,q/2)$ as observed in Section \ref{comparison}) where it is possible for connected configurations to have all atoms bonded with a substrate atom.

%\begin{remark}
%We also notice that in the case $r>1$ and $c_{S}>6 c_{F}$ properties \emph{(i)}  and \emph{(ii)} hold for any configurations $D^{\rm w}_n\subset\{(s_k+(0,r))\,:\,k\in\Zz\}$ (also not satisying \eqref{connected_wetting}).
%\end{remark}

 \begin{proposition}\label{wetting_regime2}
Let $n \in \mathbb{N}$ and $z \in Z_S^1$ with $r=1$ and $q=2$.  Any wetting configuration  $D^{\rm w}_n\subset\partial\mathcal{L}_{FS}(z)$  satisfies \PPP the following two assertions: \EEE
\begin{itemize} 
	\item[(i)] $V_n(D^{\rm w}_n)=\min{V_n(D_n)}$,
\item[(ii)] \PPP $V_n(D^{\rm w}_n)<V_n(D_n)$ \EEE for  \PPP any $D_n$ that is not a wetting configuration, % \ZZZ (I added strict inequality among all other configuration)\EEE
%\item[(ii)] $\PPP V_n(D^{\rm w}_n) <  V_n(D_n)$ \EEE for  \PPP any crystalline \EEE configuration $D_n$ with $D_n\cap\mathbb{R}\times\{r>e_S \}\neq\emptyset$\PPP, \EEE \ZZZ (we do not have strict inequality?)\EEE
\end{itemize}
if and only if  %the \emph{wetting condition} is verified
\begin{equation}\label{wetting_condition2}
c_{S}\geq 4 c_{F}.
%\begin{cases}
%c_{S}>6 c_{F} & \textrm{if $r>1$},\\
%c_{S}>5 c_{F} & \textrm{if $r\leq1$}.
%\end{cases}
\end{equation}

\end{proposition} 

\begin{proof}
The proof is the same as the proof of Proposition 3.2 in \cite{PiVe1} for the model $\mathcal{M}^0$ with $q=1$, by observing that 
\begin{equation}\label{wetting_energy}
V_n (D^{\rm w}_n)=-c_S n -2c_F(n-1)
\end{equation}
also with respect to the model $\mathcal{M}^1_{\Lambda}(z)$ for the conditions of the assertion on the wall parameters.

\end{proof}

\PPP We finally address the remaining case of $r=1$ and $q>2$ for which we notice that $\partial\mathcal{L}_{FS}$ contains separated pairs of neighboring film atoms.

 \begin{proposition}\label{wetting_regime3} 
Let $n \in \mathbb{N}$  and  $z \in Z_S^1$ with $r=1$ and $q>2$.  Any wetting configuration  $D^{\rm w}_n\subset\partial\mathcal{L}_{FS}(z)$   satisfies \PPP the following two assertions: \EEE
\begin{itemize} 
	\item[(i)] $V_n(D^{\rm w}_n)=\min{V_n(D_n)}$,
\item[(ii)] \PPP $V_n(D^{\rm w}_n)<V_n(D_n)$ \EEE for  \PPP any $D_n$ that is not a wetting configuration,  % \ZZZ (I added strict inequality among all other configuration)\EEE
%\item[(ii)] $\PPP V_n(D^{\rm w}_n) <  V_n(D_n)$ \EEE for  \PPP any crystalline \EEE configuration $D_n$ with $D_n\cap\mathbb{R}\times\{r>e_S \}\neq\emptyset$\PPP, \EEE \ZZZ (we do not have strict inequality?)\EEE
\end{itemize}
if and only if  %the \emph{wetting condition} is verified
\begin{equation}\label{wetting_condition3}
c_{S}\geq 5 c_{F}.
%\begin{cases}
%c_{S}>6 c_{F} & \textrm{if $r>1$},\\
%c_{S}>5 c_{F} & \textrm{if $r\leq1$}.
%\end{cases}
\end{equation}

\end{proposition} 

\begin{proof}

\PPP The proof is based on the same arguments employed for previous two propositions and on the following observations. 
Note that (i) easily follows from (ii) and the fact that any wetting configuration $D^{\rm w}_n$ for $r=1$ and $q>2$ has the same energy given by 
\begin{equation}\label{wetting_energy_odd}
V_n (D^{\rm w}_n)=-c_S n -c_F(n-1).
\end{equation}
if $n$ is odd, and 
$$%\begin{equation}\label{wetting_energy4}
V_n (D^{\rm w}_n)=-c_S n -c_Fn.
$$ %\end{equation}
if $n$ is even.

For the sufficiency of  \eqref{wetting_condition3} in order to prove (ii) we proceed by induction on $n$. We first notice that (ii) is trivial for $n=1$. Then, we  assume that (ii) holds true for every $k=1,\dots,n-1$ and prove that it holds also for $n$. Let $D_n$ be a crystalline configuration that is not a wetting configuration for $r=1$ and $q>2$. 
We can assume without loss of generality that $D_n\cap(\mathbb{R}\times\{s>e_S \})\neq\emptyset$ because if not, we can easily see that the energy of $D_n$ is higher than the energy of $D^{\rm w}_n$ at least by $c_S-2c_F$, which is positive by  \eqref{wetting_condition3},  since the elements in $D_n\setminus\partial\mathcal{L}_{FS}\neq\emptyset$ have at most two film bonds and no substrate bonds. Let $L$ be the last line in $\Rz\times\{s>0\}$ parallel to $\tone$ that intersects $D_n$ by moving upwards from $\Rz\times\{e_S\}$ (which exists since $D_n$ has a finite number of atoms). We can then use the same  argument used  for \cite[Eq. (36)]{PiVe1} to prove that
\begin{align*}
 V_n(D_n) & \geq  V_{n-\ell}(D_n\backslash L)-6c_F(\ell-1)-4c_F \end{align*}
and hence, by induction hypothesis if $n$ is odd (the other case being analogous), we have that
\begin{align*}
V_n(D_n) & \geq  V_{n-\ell}(D_n\backslash L)-6c_F(\ell-1)-4c_F> V_{n-\ell}(D_n\backslash L)-6c_F\ell\\
 &> -c_S (n-\ell)-c_F (n-\ell-1)-6c_F\ell\geq -c_Sn\\
  &> -c_S (n-\ell)-(c_F (n-\ell-1)+c_F\ell)-5c_F\ell\\
  &\geq -c_Sn-c_F(n-1)=V_n(D^{\rm w}_n), \\
 \end{align*}
%Therefore, we can assume that  $D_n\cap(\mathbb{R}\times\{r>e_S \})\neq\emptyset$ for which we can use the same induction argument of Proposition 3.1 in \cite{PiVe1}.  
where we used \eqref{wetting_condition3} in the last inequality and \eqref{wetting_energy_odd} in the last equality.

In order to prove the necessity of \eqref{wetting_condition3} for assertions (i) and (ii), we consider the Wulff shape with $n$ atoms in $\Rz\times\{s>\SSS e_{FS} \BBB \}$ which has energy $-6c_Fn+C\sqrt{n}$ for some constant $C>0$, and observe that  
$$-c_S n -c_Fn < -6c_Fn +C \sqrt{n} $$ 
by assertion (ii). From dividing by $n$ and letting  $n\to\infty$  we obtain \eqref{wetting_condition3}.

%To prove (i) it is easy to see that the claim is valid for $n=1,2,3$. Notice that the wetting energy is given by $(-2c_F-c_S) (n-2)-2c_F-2c_S$, for $n\geq 3$. The proof goes in the same way as the proof of Proposition \ref{wetting_regime1}. 
\end{proof}

 %\SSS \begin{remark} 
 We are now ready to characterize the wetting regime for the models $\mathcal{M}^1_{\Lambda}(z)$. We refer to \eqref{wetting_condition1}, \eqref{wetting_condition2}  and \eqref{wetting_condition3} as \emph{wetting conditions}. 
	%\end{remark} \BBB	
%\begin{remark} \label{example}
%\PPP Condition \eqref{wetting_condition2} is weaker than \eqref{wetting_condition1} because if $q=1$, then film atoms of wetting configurations can be bonded to the two film atoms at their sides in $\partial\mathcal{L}_{FS}$ (if filled) besides to their corresponding substrate atom, and Proposition \ref{wetting_regime2} \SSS shows \BBB that such configuration are preferable. 
%only of the film can be at the same time in contact with two other atoms of film and the atom of the substrate (this is valid for all atoms of the film, except the first and the last in line, which can be in contact with the atom of the substrate and one atom of the film). However, if we take that $e_S$ is not of the form $\frac{1}{n}$  (e.g. $e_S=\frac{2}{3}$) then this is again not possible and the analysis would then follow the case $\MMM q\neq 1 \EEE$  with the wetting condition given by \eqref{wetting_condition1}. 
 \BBB
%\end{remark} 	

\PPP
 \begin{theorem}[Wetting regime for $\mathcal{M}^1_{\Lambda}(z)$]\label{wetting_theorem_Z1}
For every $n\in\Nz$ any  wetting configuration $D^{\rm w}_n\in\mathcal{C}_n$ satisfies \EEE the following two assertions: 
\begin{itemize} 
	\item[(i)] $V_n(D^{\rm w}_n)=\min{V_n(D_n)}$,
	\item[(ii)] \PPP $V_n(D^{\rm w}_n)<V_n(D_n)$ \EEE for  \PPP every crystalline configuration $D_n\in\mathcal{C}_n$ that is not a wetting configuration,
%\item[(ii)] $\PPP V_n(D^{\rm w}_n) <  V_n(D_n)$ \EEE for  \PPP any crystalline \EEE configuration $D_n$ with $D_n\cap\mathbb{R}\times\{r>e_S \}\neq\emptyset$\PPP, \EEE \ZZZ (we do not have strict inequality?)\EEE
\end{itemize}
if and only if  %the \emph{wetting condition} is verified
%c_{S}\geq 6 c_{F} if $z \in Z_S^1$ with either $q\neq1$, or $q=p=1$ and $z_2\neq1$
%c_{S}\geq 5 c_{F} if $z \in Z_S^1$ with $z_2=1$ and $q=p=1$
%c_{S}\geq 4 c_{F} if $z \in Z_S^1$ with $q=1$ and $p\neq1$
\begin{equation}\label{wetting_condition_total} 
\begin{cases}
c_{S}\geq 6c_{F} & \text{if $z \in Z_S^1$ with $r>1$  \emph{(}and so $q>2$\emph{)},}\\
c_{S}\geq 5c_{F} & \text{if $z \in Z_S^1$ with $r=1$ and $q>2$}, \\
c_{S}\geq 4c_{F} & \text{if $z \in Z_S^1$ with $r=1$ and $q=2$}. 
\end{cases}
\end{equation}
%\begin{equation}\label{wetting_condition_total} 
%\begin{cases}
%c_{S}\geq 6c_{F} & \text{if either $z \in Z_S^0$ and $q\neq 1$, or $z \in Z_S^1$ and $z_2\neq 1$},\\
%c_{S}\geq 5c_{F} & \text{if $z\in Z_S^1$ and $z_2=1$}, \\
%c_{S}\geq 4c_{F} & \text{if $z \in Z_S^0 $ and $q=1$}. 
%\end{cases}
%\end{equation}
\FFF In particular, for the necessity of \eqref{wetting_condition_total} it is enough assertion (i), and more specifically that 
%or more precisely it is enough that 
%. Notice that such necessity 
%The reverse implication in the Theorem \ref{wetting_theorem} 
%should be understood  in the following way: If 
there exists an increasing subsequence $(n_k)_{k \in \mathbf{N}}$ such that $V_{n_k}(D^{\rm w}_{n_k})=\min{V_{n_k}(D_{n_k})}$ holds for every $n_k$. \EEE	
\end{theorem}

\begin{proof}
The first assertion directly follows from Propositions  \ref{wetting_regime1}, \ref{wetting_regime2}, and \ref{wetting_regime3} for $z\in Z^1_S$, while the second assertion is a direct consequence of the limiting procedure in the proofs of the necessity of the wetting conditions of such results. % in the proofs of Propositions \ref{wetting_regime1}, \ref{wetting_regime2}, and \ref{wetting_regime3} for $z\in Z^1_S$.  %that for the necessity of the corresponding wetting conditions it is enough assertion (i) or,  more precisely, it is enough that 	there exists an increasing subsequence $(n_k)_{k \in \mathbf{N}}$ such that (i) holds for every $k \in \mathbb{N}$, while the wetting conditions are sufficient for (i) and (ii) for every $n \in \mathbb{N}$.  		

\end{proof}

%\section{Compactness}\label{sec:compactness}

\PPP We now move on studying the dewetting regime for the model $\mathcal{M}^1_{\Lambda}(z)$. Therefore,  in the remaining part of this section we work  only under the assumption
\begin{equation}\label{dewetting_condition_Z1} 
\begin{cases}
c_{S}< 6c_{F} & \text{if $z \in Z_S^1$ with $r>1$  (and so $q>2$),}\\
c_{S}< 5c_{F} & \text{if $z \in Z_S^1$ with $r=1$ and $q>2$}, \\
c_{S}<4c_{F} & \text{if $z \in Z_S^1$ with $r=1$ and $q=2$}. 
\end{cases}
\end{equation}

We begin by \PPP establishing a lower bound uniform for \PPP every \EEE $x\in D_n\cap\partial \mathcal{L}_{FS}$ in terms of $c_F$ and $c_S$ \PPP of the \EEE strip energy $E_{\rm strip}(x)$ defined in Section \ref{local_energy}. %To this aim, \PPP we need to distinguish the case $q= 1$ from $q\neq1$ as already done in Section \ref{sec:wetting} because of the different contributions in $E_{\rm strip}(x)$ of the substrate interactions. \EEE

 \begin{lemma}\label{strip_local}
 If \eqref{dewetting_condition_Z1}  holds, then 
$$%\begin{equation}\label{strip_bound}
E_{\rm strip}(x)\geq %-6c_f \#\mathcal{S}(x) +
\Delta_{\rm strip}
$$%\end{equation}
with
\begin{equation}\label{delta_strip}
\Delta_{\rm strip}:=\begin{cases} 
6c_{F}-c_{S} & \text{if $z \in Z_S^1$ with $r>1$  \emph{(}and so $q>2$\emph{)},}\\
5c_{F}-c_{S}  & \text{if $z \in Z_S^1$ with $r=1$ and $q>2$},\\
4c_{F}-c_{S}  & \text{if $z \in Z_S^1$ with $r=1$ and $q=2$}, 
\end{cases}
\end{equation}
for every $x\in D_n\cap\partial \mathcal{L}_{FS}$.
 \end{lemma}
   \begin{proof}
 %Assume $\MMM q\neq 1 \EEE$ {\color{pink} and $\tilde x \neq x$ (in either of these cases the proof is simpler)}.  
\FFF Fix \EEE $x\in D_n\cap\partial \mathcal{L}_{FS}$. %such that $v^1(x)\neq0$. 
We begin by observing that the strip center $x$ surely misses the bonds with the atoms missing at the 2 positions $x-\ttwo+k\tone$ for $k=0,1$ \FFF as shown in Figure \ref{fig:strip}\EEE. Furthermore, either $x$ misses the bond with $x_{-}$, or $x_{-}\in {\color{black} D_n} $  and $x_{-}$ misses the bonds with the 2 positions $x-\ttwo+k\tone$ for $k=-1,0$ (which in the strip energy are counted with half weights). We can reason similarly for $x_{+}$. Therefore, by the definition of energy of the low strip $E_{\rm strip,below}$,  we obtain that 
$$
E_{\rm strip,below}\FFF \geq\EEE\begin{cases}
4c_{F}-c_{S} & \text{if $z \in Z_S^1$ with $r>1$ (and so $q>2$),}\\
3c_{F}-c_{S}  & \text{if $z \in Z_S^1$ with $r=1$ and $q>2$}, \\
2c_{F}-c_{S}   &  \text{if $z \in Z_S^1$ with $r=1$ and $q=2$}.
\end{cases}
$$

%if $\MMM q\neq 1 \EEE$, 
%we {\color{black} $4c_f$}. {\color{black} In the case when $e_S \leq 1$ we have the contribution of at least $2c_F$}. 

The analysis of $E_{strip, above}$ is exactly the same as in \cite[Lemma 4.1]{PiVe1} since  $E_{strip, above}$ for $\mathcal{M}^1z$ for $z\in Z_S^1$ is defined as for $\mathcal{M}^0$. In particular,  the terms related to the triple $\tilde{x}$, $\tilde{x}_+$, and $\tilde{x}_-$ give always a contribution in the strip energy $E_{\rm strip}$ of at least $2c_F$.

   \end{proof}

In the following result by employing a similar non-local \emph{strip argument} of \cite[Lemma 4.2]{PiVe1} related to the more involved definition of the strip $\mathcal{S}$ of Section \ref{local_energy} we can show that the energy $V_n(D_n)$ of any crystalline configuration $D_n$ is bounded from below by  $-6c_Fn$ plus a positive \emph{deficit} due to the contribution of the boundary of $D_n$, i.e., where atoms have less than 6 film bonds \PPP and \EEE could have a bond with the substrate.  
  %In the following lemma we improve this lower bound by considering not only that atoms at the boundary $\partial D_n$ have at least one bond missing, but also that atoms at $\partial D_n\cap\partial \mathcal{L}_F$ may be bonded to substrate atoms.

 \begin{lemma}\label{strip}
 If \eqref{dewetting_condition_Z1} holds, then there exists $\Delta>0$ such that
\begin{equation}\label{Vn_lower_bound}
 V_n(D_n)\geq -6c_Fn\,+\,\Delta\#\partial D_n
\end{equation}
 for every crystalline configuration $D_n\subset\mathcal{L}_F$. \PPP Furthermore, the following two assertions are equivalent:
 \begin{itemize}
  \item[(i)] There exists a constant $C>0$ such that $\# \partial D_n \leq C \sqrt n$ for every $n \in \mathbb{N}$,
 \item[(ii)]  There exists a constant $C'>0$ such that $E_n (\mu_{D_n}) \leq C'$ for every $n \in \mathbb{N}$.
 \end{itemize}

  \end{lemma}

 \begin{proof}
The proof is analogous to the proof of \cite[Lemma 4.2]{PiVe1} where instead of \cite[Lemma 4.1]{PiVe1} we apply Lemma \ref{strip_local}. In fact, by \eqref{Eloc} and \eqref{Estrip} we observe that also by  the careful choice of the weights \eqref{Eweight_below1}-\eqref{Eweight_below3} in \eqref{Estrip_below}, besides of the choice of weights  \eqref{weigths_above} in \eqref{Estrip_above}, we have that 
\begin{align}
6c_Fn\,+\,
V_n(D_n)\,&=\, \sum_{x\in D_n}\left(\sum_{y\in {\color{black} D_n}\setminus\{x\}} v_{FF}(|x-y|)+ 6c_F
\right)\,+\,\sum_{x\in D_n}  v^1(x)\notag\\
&=\, \sum_{x\in D_n} E_{\rm loc}(x)\,+\,\sum_{x\in D_n}  v^1(x)\notag\\
&\geq \, \sum_{x\in D_n\cap {\color{black} \partial \mathcal{L}_{FS}}} E_{\rm strip}(x)\,+\,\sum_{x\in D_n\setminus\mathcal{S}({\color{black} \partial \mathcal{L}_{FS}})} E_{\rm loc}(x). \label{energy_strip}
\end{align}
with 
\begin{equation}\label{layer_strip}
\mathcal{S}(\partial \mathcal{L}_{FS})=\mathcal{S}_{\MMM D_n\EEE}({\color{black} \partial \mathcal{L}_{FS}}):=\{y\in\mathcal{S}(x)\,:\,x\in D_n\cap{\color{black} \partial \mathcal{L}_{FS}} \},
\end{equation}
where we used that  $v^1(x)=0$ for every $x\in D_n\setminus\partial \mathcal{L}_{FS}$, since  the local energy $E_{\rm loc}(x)$ of every film atom in $D_n\cap {\color{black} \partial \mathcal{L}_{FS}}$  is counted {\color{black} at most once}.

 \end{proof}

\PPP In view of the previous lower bound for the energy of a configuration $D_n$ we are  able to compensate the negative contribution coming at the boundary from the interaction with the substrate obtaining the following compactness results.  %\PPP We notice that to achieve compactness the negative contribution coming at the boundary from the interaction with the substrate needs to be compensated. This is not trivial, e.g., in the case $6c_F>c_S>4c_F$, where atoms $x$ of configurations on $\partial\mathcal{L}_{FS}$ have one bond with a substrate atom and at least two bonds with film atoms  missing. A way to solve the issue is to look for extra positive contributions from other atoms in the boundary. However, just looking for neighboring atoms might be not enough, e.g., in the case with \SSS $e_{FS}=e_S=2$ or  $e_{FS}=e_S=\frac{2}{3}$\BBB.  The issue is solved in the proof of the following compactness result by introducing a new non-local argument called the \emph{strip argument} that involve looking at the whole strip $\mathcal{S}(x)$.

 \begin{proposition} \label{zadnje} 
\PPP Assume that \eqref{dewetting_condition_Z1} holds. Let  $D_n\in\mathcal{C}_n$ be almost-connected configurations such that
\begin{equation}\label{upperbound} 
V_n(D_n)\leq -6c_Fn +Cn^{1/2}
\end{equation}
for a constant $C>0$. Then there exist an increasing sequence $n_r$, $r\in\Nz$, and a measure $\mu\in \mathcal{M}(\Rz^2)$ with $\mu\geq0$ and $\mu(\Rz^2)=1$  such that   $\mu_{r}\stackrel{*}{\rightharpoonup} \mu$ in $\mathcal{M}(\Rz^2)$, where $\mu_{r}:=\mu_{D_{n_r}\PPP(\,\cdot\,+a_{n_r})}$ for some translations $a_n\in\Rz^2$ \EEE \emph{(}see  \eqref{empiricalmeasures} for the definition of  the empirical measures $\mu_{D_{n_r}}$\emph{)}. \PPP Moreover, if $D_n\in\mathcal{C}_n$ are minimizers of $V_n$ in $\mathcal{C}_n$, then we can choose $a_n=t_n\tone$ for integers $t_n\in\Zz$. \EEE
 \end{proposition} 
 
 \begin{proof}
 The proof follows the same strategies employed for the proof of \cite[Proposition 4.3]{PiVe1}, where instead of applying \cite[Lemma 4.2]{PiVe1} we apply the previously proven Lemma \ref{strip}. 
 \end{proof}

\PPP  We are now ready to state the following compactness result  which is a consequence of Proposition \ref{zadnje}  and of the definition of transformed configuration \PPP $\mathcal{T}(D_n)$ associated to a configuration $D_n$  provided in Definition \ref{transformation} (see also  (see also \cite[Theorem 4.1]{PiVe1} and  \cite[Theorem 1.1]{Yeung-et-al12}).  %and the same arguments by replacing  \cite[Proposition 3.2]{Yeung-et-al12} with Proposition \ref{zadnje}.  \EEE

%\PPP \SSS We \BBB conclude the section with compactness results for sequences of almost-connected configuration (see Section \ref{sec:almost_connected} for the definition). We remind the reader that by the trasformation defined in Definition \ref{transformation} for any configuration $D_n$ there exists the almost-connected configuration $\widetilde{D}_n$ such that $V_n(\widetilde{D}_n)\leq V_n(D_n)$. \EEE

\EEE

 \begin{theorem}[Compactness] \label{compactnesstheorem}
 Assume that \eqref{dewetting_condition_Z1} holds.  Let $D_n\in\mathcal{C}_n$ be configurations  satisfying \eqref{upperbound} and \PPP let \EEE $\mu_{n}:=\mu_{\mathcal{T}(D_n)}$ \PPP be \EEE the empirical measures associated to the transformed configurations \PPP $\mathcal{T}(D_n)\in\mathcal{C}_n$ associated to $D_n$  by Definition \ref{transformation}. \EEE  Then, \PPP up to translations $($i.e., up to replacing $\mu_{n}$ by  $\mu_{n}(\cdot+a_n)$ for some $a_n\in\Rz^2$$)$ and a passage to a non-relabelled subsequence,  $\mu_{n}$ converges \SSS weakly* \BBB  in $\mathcal{M}(\mathbb{R}^2)$ to a measure $\mu\in\mathcal{M}_W$, where	$\mathcal{M}_W$ is defined in \eqref{M_Wulff}. \PPP Furthermore, if $D_n\in\mathcal{C}_n$ are minimizers of $V_n$ in $\mathcal{C}_n$, then we can choose $a_n=t_n\tone$ for integers $t_n\in\Zz$. %\ZZZ(I added specification of the direction $\tone$ for the translations of minimizers)\EEE\EEE
\end{theorem}

%\RRR 

The next results allows to overcome the issue of compactness for not almost-connected configurations (without using associated transformed configurations as in Theorem \ref{compactnesstheorem}),  as it shows that mass is preserved in the limit by carefully selecting connected components of minimizers.

\begin{theorem}[Mass conservation] \label{connectness_Z1} 
\PPP Assume that \eqref{dewetting_condition_Z1} holds. If $ \widehat{D}_n$ are  minimizers of $V_{n,\Lambda}$ among all crystalline configurations in $\mathcal{C}_n$, i.e.,  
$$V_{n,\Lambda}(\widehat{D}_n)=\min_{D_n\in\mathcal{C}_n}{V_{n,\Lambda}(D_n)},$$
and we select for every $\widehat{D}_n$ a connected component $\widehat{D}_{n,1}\subset \widehat{D}_n$ with  largest cardinality, then
$$
\lim_{n \to \infty}{\mu_{\widehat{D}_n} (\widehat{D}_{n,1})}=1, 
$$
 %\ZZZ(I am a bit surprise that we can write this in general and not for specific translations of ${\color{pink}\widehat{D}_{n,1}}$, but checking the proof it seems right) \PPP 
where $\mu_{\widehat{D}_n}$ are the empirical measure associated to $\widehat{D}_n$ defined by \eqref{empiricalmeasures}. \EEE
\end{theorem}	

%\RRR 
The proof of  Theorem \ref{connectness_Z1} is exactly the same as for the analogous   \cite[Theorem 2.3]{PiVe1} and it actually depends on the first two assertions of the following theorem (whose statement we postpone below as it requires some adjustment in the proof): first one establishes the $\Gamma$-convergence result (see Assertion 1 of Theorem \ref{thm:convergence_minimizers_Z1}), which together  with Theorem \ref{compactnesstheorem} implies the existence of minimizer  (see Assertion 2 of Theorem \ref{thm:convergence_minimizers_Z1}) for the  limiting functional $I_{\infty,4}$ defined in \eqref{I4}, then, one shows that it is impossible for a sequence  of  minimizers of $E_{n,\Lambda,x_F}$ to have a subsequence of disconnected  components with significant mass, since this  would imply that there exists a disconnected minimizer of $I_{\infty,4}$, which it is an absurd by scaling arguments.
\BBB
%Since  for the connected sets with bounded energy 	$E_{n,\Lambda,x_F}$ necessarily satisfy compactness in weak* topology we obtain the statement 3. of Theorem \ref{thm:convergence_minimizers_Z1}. 

 %Theorem \ref{thm:convergence_minimizers_Z1} establishes $\Gamma$-convergence statement and existence of minimizers for the limit problem as well as convergence of the minimizing sequence to the minimizer of the limit functional. 

%\BBB

\begin{theorem}[Convergence of Minimizers]\label{thm:convergence_minimizers_Z1}
	\PPP Assume  that \eqref{dewetting_condition_Z1} holds. The following statements hold:
	\begin{itemize}
	\item[1.] The functional 
	\begin{equation}\label{converging_energies00}
	E_{n,\Lambda,x_F}:= n^{-1/2}(I_{n,\Lambda,x_F}+6c_F n),
	\end{equation}
	where $I_{n,\Lambda,x_F}$ is defined by \eqref{radon_functional}, $\Gamma$-converges with respect to the weak* convergence of measures to the functional $I_{\infty,4}$ defined by
		\begin{equation}\label{I4}
	I_{\infty,4}(\mu):=\begin{cases}  \mathcal{E}_{\sigma_4}(D_\mu), &  \text{if \FFF there exists \EEE $D_\mu\subset\Rz^2\setminus S$  set of finite perimeter}\\
	& \hspace{15ex} \text{with  $|D_\mu|=1/\rho$ such that $\mu=\rho\chi_{D_\mu}$,}
\\%\text{if $\mu\in\mathcal{M}_W$},\UUU\\
	+\infty, &\text{otherwise,}\\
	\end{cases}
	\end{equation}
	for every $\mu\in\mathcal{M}(\Rz^2)$,
where $\rho:=2/\sqrt{3}$ and $\mathcal{E}_{\sigma_4}$ is defined in \eqref{continuum} for an adhesivity $\sigma:=\sigma_4$ where
\begin{equation}\label{sigma4}
\sigma_{4}:=2c_F-\displaystyle\frac{2c_S}{q}.
\end{equation}

	 %, where 
%\begin{align*}
%	\mathcal{M}_W:=\bigg\{\mu\in\mathcal{M}(\Rz^2)\ :\ \text{$\exists$ $D\subset\Rz^2\setminus S$  set of finite perimeter, bounded, and with } & |D|=\frac{\sqrt{3}}{2}\\
%	& \hspace{-15ex} \text{such that $\mu=\frac{2}{\sqrt{3}}\chi_D$}\bigg\}.
%	\end{align*}
	\item[2.]  The functional $I_{\infty,4}$ admits a minimizer in 
	\begin{align}
\mathcal{M}_W:=\bigg\{\mu\in\mathcal{M}(\Rz^2)\ :\ \text{$\exists$ $D\subset\Rz^2\setminus S$} & \hspace{2ex} \text{set of finite  perimeter, bounded }\notag\\
\hspace{8ex} \text{with }  |D|=\frac{1}{\rho},\label{M_Wulff}	& \hspace{2ex} \text{ and such that $\mu=\rho\chi_D$}\bigg\}. 
	\end{align}
	\item[3.]  Every sequence $\mu_n\in\mathcal{M}_n$ of minimizers of $E_{n,\Lambda,x_F}$ %of crystalline configurations $D_n^m\subset\mathcal{L}_F$  of the functional \UUU
%	\begin{equation}\label{converging_energies}
%	E_n:= n^{-1/2}(I_n+6c_F n)
%	\end{equation}
admits, up to translation %in the direction $\tone$ $($i.e., up to replacing $\mu_n$ with $\mu_n(\cdot+c_n\tone)$ for  chosen fixed integers $c_n\in\Zz$$)$, 
a subsequence  converging  with respect to the weak* convergence of measures to a minimizer of $I_{\infty,4}$ in $\mathcal{M}_W$. 
	\end{itemize}
	\end{theorem}
	
	To prove Theorem \ref{thm:convergence_minimizers_Z1} we follow the proof of \cite[Theorem 2.4]{PiVe1}, which consists in proving the lower and the upper bound  established in \cite[Theorems 5.3 and 6.1]{PiVe1}, respectively. The upper bound can be proven in the exact same way as done in \cite[Theorem 6.1]{PiVe1}:  the set $D \subset \mathbb{R}^2 \backslash S$ of finite perimeter satisfying $|D|=1/\rho$ is approximated first by smooth bounded sets, then by polygons, then by polygons with vertices on the lattice, and finally by the sequence of configurations $(D_n)_{n \in \mathbb{N}}$ such that $D_n \in \mathcal{C}_n$, $\mu_{D_n} \stackrel{*}{\rightharpoonup} \rho \chi_D$ and $I_{n,\Lambda,x_F}(\mu_{D_n}) \to I_{\infty, 4}(\rho \chi_D)$. 
	
The lower bound can be proven by an adaptation of the arguments of the proof of \cite[Theorem 5.3]{PiVe1}. We first describe the overall strategy and then highlight the main differences in such arguments in the proof below.

The strategy consists in associating  to each configuration $D_n$ an auxiliary set $H_n'$ such that $\chi_{H_n'}$ converges in $BV_{\rm loc}$ to  $\chi_D$ (see also \cite{Yeung-et-al12}) in order to transfer the problem on the surface energy functionals analyzed in, e.g., \cite{BA}. Furthermore, by fixing  $\delta>0$ we reduce the analysis  to the region  $\{x \in \mathbb{R}^2:0\leq x_2 \leq\delta\}$, since outside of it one can simply use  Reshetnyak Lower-Semiconinuity Theorem. 	
For the points $x \in \partial^*D \cap \{x_2=0\}$ one obtains the lower bound by using the fact that for $n$ large enough the upper half of an arbitrary small square with  center in $x$ and sides parallel to the coordinate axes is ``approximately filled'' by the set $H_n'$ and hence, containing in the worst case all the atoms that have substrate neighbors, whose energy density is $-2c_S/q$. However, also the other atoms without substrate neighbors must be contained, and every  boundary atom in such  upper half of the arbitrary small square on every line perpendicular to $x_2=0$ brings the energy density of $2c_F$. Furthermore, one proves that for each $x \in  \{x_2=0\} \backslash \partial^*D $   there is no negative energy contribution coming from the atoms  with substrate neighbors that are contained in the upper half of an arbitrary small square centered at $x$ by employing, under the dewetting condition \eqref{dewetting_condition_Z1},  the continuum analogue of the strip argument described in Lemma \ref{strip_local} to exploit that $H_n'$ ``vanishes'' in such arbitrary small square around $x$ for $n$ large enough. Finally, $\delta$ is sent to zero obtaining the lower bound.

\begin{proof} 
	
We now detail the adaptations needed in the arguments of the strategy used in  \cite[Theorem 5.3]{PiVe1} (and described above) to prove the lower bound. Such adaptations are needed exclusively for the case with  $r=1$ and $q>2$. Furthermore, we refer only to the modifications in Step 1 of the proof of \cite[Theorem 5.3]{PiVe1} since the modifications of Step 2 are analogous, and Step 3 is exactly the same. 
We need to directly refer to the notations introduced in 	\cite[Section 5]{PiVe1} to describe  the two needed adaptations.

The first modification relates to \cite[Eq. (69)]{PiVe1} that in our setting becomes  %in the proof of \cite[Theorem 5.3]{PiVe1} we can conclude 
		\begin{eqnarray*}
		2c_F \zeta_1^M(y')-c_S \zeta_2^M(y')&=& \lim_{\varepsilon \to 0} \frac{2c_F \kappa_1^M(Q_{\varepsilon}(y'))-c_S \kappa_2^M(Q_{\varepsilon}(y'))}{\varepsilon} \\ &=& \lim_{\varepsilon \to 0}  \lim_{n \to \infty}  \frac{\kappa_n(Q_{\varepsilon} (y'))}{\varepsilon} \\
		&\geq & {\OOO 2/q (2c_F-c_S)+(\alpha-2/q)2c_F}\varepsilon, 
	\end{eqnarray*}   
	since  in our setting 
	$$k_1(n) \leq 2\left\lceil \frac{\sqrt{n}\eps}{q} \right\rceil$$ because  in a period of $q$ film atoms in  $\partial\mathcal{L}_F$at most two of them can be bonded with the substrate.

	The second modification is required to show that points that do not belong to the set $\partial^* D$ do not bring negative contribution to limiting energy of the type $\mathcal{E}_{\sigma}$. To this end, we need to redefine the set 	$\widetilde K_{n,\varepsilon}(y')$ with
	$$
	\widetilde K_{n,\varepsilon}(y'):= \left\{a_0\in K_{n,\varepsilon}(y')\,:\, \text{$\FFF\exists \EEE i,j \in\{-1,0,1\}$ such that}\, |\widetilde O_{n}^{a_i,j}\cap H_n'\cap Q_{\varepsilon} (y')|>\frac{\varepsilon}{8\sqrt n}
	\right\}
	%\frac{\partial \mathcal{L}_{FS}}{\sqrt{n}}\cap Q_{\varepsilon} (y')
	$$
	for $y'$ being an arbitrary point in $\{x_2=0\}\setminus \partial^* D$, so that instead of  \cite[Eq. (71)]{PiVe1} (again as the consequence of \cite[Eq. (70)]{PiVe1}) we conclude that
%Instead of (71) in the proof (again as the consequence of (70)) we easily conclude that 
	 \begin{equation*} 
	\#\FFF \widetilde K_{n,\varepsilon}(y') \EEE \leq 16\alpha \varepsilon \sqrt{n}+2,
	\end{equation*}  
	from which it follows that 
	\begin{equation} \label{eqivan24} 
	\sum_{a_0 \in \widetilde{K}_{n,\varepsilon}(y')}\kappa_{n}^M (\widetilde O_n^{a_0}) \geq -|2c_F-c_S|(16\alpha \varepsilon \sqrt{n}+2) \frac{1}{\sqrt{n}}
	\end{equation} 
	 (analogously as \cite[Eq. (72)]{PiVe1} was deduced)	. It remains to prove that in view of \eqref{eqivan24}  any $a_0 \in K_{n,\varepsilon}(y')\backslash\widetilde{K}_{n,\varepsilon}(y')$ does not give, for $n$ large enough, negative contribution to the limiting energy. In this regard, we notice that due to the different wetting condition of our setting we obtain 
		\begin{equation} \label{novelty1}  
	\sum_{a_0 \in K_{n,\varepsilon}(y')\backslash\widetilde{K}_{n,\varepsilon}(y')}\kappa_{n}^M (\widetilde O_n^{a_0}) \geq (5c_F-c_S) \frac{\# (K_{n,\varepsilon}(y')\backslash\SSS \widetilde{K}_{n,\varepsilon}(y') \BBB)}{\sqrt{n}}\geq 0
	\end{equation}
	in place of  \cite[Eq. (74)]{PiVe1}. We now analyze only the case when $a_1$ or $a_{-1}$  being the film neighbors  of $a_0$ in $\partial\mathcal{L}_{F}/\sqrt{n}$, 
		 belong to $(D_n \cap \partial \mathcal{L}_{FS})/\sqrt{n}$, since if this is not the case the analysis goes in the completely same way as in the proof of \cite[Theorem 5.3]{PiVe1}. 
		 Without loss of generality we assume that   $a_1 \in (D_n\cap \partial \mathcal{L}_{FS})/\sqrt{n}$ (and so $a_{-1} \not\in (D_n\cap \partial \mathcal{L}_{FS})/\sqrt{n}$).
		 
	Instead of \cite[Eq. (73)]{PiVe1} we just need to show that
		\begin{equation} \label{eqivan2221} 
	\mathcal{H}^1(\partial H_n'\cap Q_{\varepsilon} (y') \cap (\widetilde O_n^{a_0}\cup \widetilde O_n^{a_1})) \geq \FFF \frac{5}{\sqrt{n}}
	\end{equation}
	by considering the following three options: 
	\begin{enumerate}
		\item both  of the strips  $\widetilde O_{n}^{a_0,-1}$ and $\widetilde O_{n}^{a_1,+1}$  have empty intersection with $H_n'$;
		\item  one of the strips  $\widetilde O_{n}^{\FFF a_0,-1\EEE}$ and $\widetilde O_{n}^{a_1,+1}$   has empty intersection with $H_n'$;
		\item  none of the strips  $\widetilde O_{n}^{a_0,-1}$ and $\widetilde O_{n}^{a_1,+1}$   has empty intersection with $H_n'$. 
	\end{enumerate} 
For simplicity we address only the first options (being the other analogous to options (2) and (3) in Step 1 of the proof of \cite[Theorem 5.3]{PiVe1}):  Both $a_0$ and $a_1$ miss both  either four or five film neighbors. In the case they miss five neighbors (and thus all film neighbors besides each other) we easily obtain 
	\eqref{eqivan2221} by analyzing the associated set $H_n'$. 
	In the case when they miss four film neighbors it follows that $a_0+\ttwo \in D_n$, which misses two neighbors and thus \eqref{eqivan2221} can also be easily deduced. %The cases (2) and (3) can be analyzed in the similar way (see also the proof of \cite[Theorem 5.3]{PiVe1}). 
%	The other steps are proven in an analogous way. 

This conclude the modifications of the arguments. 
	\end{proof}

\section{Proofs of main results}\label{sec:main_results}

In order to prove the main results contained in  Section \ref{sec: main_result} we need to address the case $x_F\neq x_F^0$, which is related neither to models $\mathcal{M}^0_{\Lambda}(z)$ nor to models $\mathcal{M}^1_{\Lambda}(z)$. 

We start by considering models $\mathcal{M}_{\Lambda}(x_F, z)$ defined for vector parameters $\Lambda\in(\Rz^+)^3$ and $z=(p,q,0)\in Z_S^0$ such that ${\rm dist}(\partial\mathcal{L}_F, \partial\mathcal{L}_S)=e _{FS}$.
%\RRR Recall that we assumed $x_F \in \partial \mathcal{L}_{FS}$
 The next result allows us to overcome this issue by using the notion of  model equivalence  introduced in Section \ref{comparison}, by proving that such models $\mathcal{M}_{\Lambda}(x_F, z)$ are equivalent  either to $\mathcal{M}^0_{\Lambda'}(z')$ or to $\mathcal{M}^1_{\Lambda''}(z'')$ for proper choices of $\Lambda',\Lambda''\in(\Rz^+)^3$, $z'\in Z_S^0$ and $z''\in Z_S^1$.

\begin{proposition}\label{equivalence_Z1}
Let $\Lambda=(e_{FS},c_F,c_S)\in(\Rz^+)^3$, $x_F\in \Rz^2\setminus\overline{S}$, and $z=(p,q,0)\in Z_S^0$ such that \eqref{firstinterface2} holds. Then the model $\mathcal{M}_{\Lambda}(x_F, z)$ is equivalent to %$\mathcal{M}^0_{\Lambda'}(z')$ or to $\mathcal{M}^1_{\Lambda''}(z'')$ for proper choices of $\Lambda',\Lambda''\in(\Rz^+)^3$, $z'=\in Z_S^0$ and $z''=\in Z_S^1$. More precisely, the equivalent model is
\begin{itemize}
\item[-]  $\mathcal{M}^0_{\Lambda'}(1,q)$ with $\Lambda'=(e_{FS},c_F,2c_S)$, if $\mathcal{M}_{\Lambda}(x_F, z)\in C_1$;
\item[-]  $\mathcal{M}^0_{\Lambda}(1,q)$, if $\mathcal{M}_{\Lambda}(x_F, z)\in C_2$;
\item[-]  $\mathcal{M}^0_{\Lambda}(1,q/2)$, if $\mathcal{M}_{\Lambda}(x_F, z)\in C_3$;
\item[-]  $\mathcal{M}^1_{\Lambda}(q,r)$ with
$$
r:=\begin{cases}
s &\textrm{ if } s <\frac{q}{2}, \\
q-s &\textrm{ if } s >\frac{q}{2},
\end{cases}
$$
 if $\mathcal{M}_{\Lambda}(x_F, z)\in C_4$.
\end{itemize}

\end{proposition}

\begin{proof}
In the following we denote by \emph{film site} any site of the film lattice. We proceed by treating separately the following different situations, which depend on the relative positioning of the lattices $\mathcal{L}_F(x_F)$ and $\mathcal{L}_S(z)$:
\begin{itemize}
\item[1)] There exists a film site $a \in \partial \mathcal{L}_{FS}(x_F,z)$ that is bonded with (at least) two substrate atoms in $\partial \mathcal{L}_S(z)$. We notice that every film site can have at most two substrate neighbors  (which then have to also be mutually neighbors) as directly follows  from \eqref{firstinterface2}. 
\item[2)] Every film site in $\partial \mathcal{L}_{FS}(x_F,z)$ is bonded with exactly zero or exactly one substrate atoms. For this situation we further distinguish the following cases:\begin{itemize}
\item[2a)] there exists at least a substrate atom that is bonded with (at least) two film sites;
\item[2b)] every substrate atom is bonded with exactly zero or exactly one film site.
\end{itemize}
\end{itemize}  

In each of the above situation we will identify the corresponding category $C_i$ for $i=1,\dots,4$ and the equivalence class for the models in such situations.

We begin by analyzing  1)  First we observe  that   every film site that has a substrate  neighbor needs to actually have  two substrate neighbors (which then  necessarily need to be mutual  neighbors). In fact, as a consequence of the fact that both $\partial \mathcal{L}_S$ and $\partial \mathcal{L}_{F}$ are $q$ periodic there exist $a_1,a_2 \in \partial\mathcal{L}_{FS}(x_F,z)$ such that $a_1$ is bonded with $s_1 \in \partial \mathcal{L}_S(z)$ and $s_1'=s_1+(e_S,0) \in \partial \mathcal{L}_S(z)$, and $a_2$ is bonded with $s_2 \in \partial \mathcal{L}_S(z)$. We denote by $s_{a_1}$ and $s_{a_2}$ the projections  on $\partial \mathcal{L}_S(z)$ of $a_1$ and $a_2$, respectively,  and let $s_2':=s_2+(e_S,0)\in \partial \mathcal{L}_S(z)$ and $s_2'':=s_2-(e_S,0) \in \partial \mathcal{L}_S(z)$. We have that ${\rm dist}(a_{\alpha},s_{\alpha})=e_{FS}$, ${\rm dist} (a_{\alpha},s_{a_{\alpha}})=\sqrt{(e_{FS})^2-(e_{S}/2)^2}$, for $\alpha=1,2$, and ${\rm dist}(a_{1},s'_{1})=e_{FS}$. Thus, we deduce that the triangle whose vertices are $s_2,s_{a_2},a_2$ is either the translation of the triangle whose vertices are $s_1, s_{a_1}, a_1$, or the translation of the triangle $s_1',s_{a_1},a_1$. In the former case we deduce that ${\rm dist}(a_2,s_2')=e_{FS}$, while in the latter case we have   ${\rm dist}(a_2,s_2'')=e_{FS}$, which gives the claim that $a_2$ has also two substrate neighbors.

 Let us now denote by $a_1, a_2 \in  \partial \mathcal{L}_{FS}(x_F,z)$, say $a_2$ on the right of $a_1$,  two closest film sites in $\partial \mathcal{L}_{FS}(x_F,z)$. Since ${\rm dist}(a_1,a_2)$ is an integer multiple of $e_S$, it has also  to be a multiple of $q$. On the other hand since the reference lattices are $q\tone$ periodic we have that ${\rm dist}(a_1,a_2)\leq q$ and hence, we conclude ${\rm dist}(a_1,a_2)=q$ and
$$\partial \mathcal{L}_{FS}(x_F,z)=\{ x_F+k \, : \, k \in q\mathbb{Z} \}. $$
 From this we deduce that in any model $\mathcal{M}_{\Lambda}(x_F, z)$ satisfying 1. belongs to class $C_1$ and is equivalent with the model $\mathcal{M}^0_{\Lambda'}(1,q)$ with $\Lambda':=(e_{FS},c_F,2c_S)$.

We now pass to the situation 2)  and in particular to  2a) and easily deduce that in this case all substrate atoms are connected with exactly zero or exactly two film (neighboring) atoms, and if $s_1,s_2 \in \partial \mathcal{L}_S(z)$ are two closest substrate atoms that are bonded with film sites, then ${\rm dist} (s_1,s_2)=q$ and the set of all substrate atoms that are bonded with film sites is given by
 $$\{s_1+k\,:\, k\in q\mathbb{Z}\}.$$
For 2a) it remains to determine the corresponding category $C_i$ for $i=1,\dots,4$ find the equivalence class:
\begin{itemize}
\item[-] if $q=1$ then $\partial \mathcal{L}_{FS}(x_F,z)=\partial \mathcal{L}_{F}(x_F)$ and every film site in $\partial \mathcal{L}_{FS}(x_F,z)$ is bonded with exactly two substrate atoms.
Moreover, since \eqref{firstinterface2} holds,  every substrate atom in $\partial \mathcal{L}_S(z)$ is bonded with two film sites.
This implies  that the model $\mathcal{M}_{\Lambda}(x_F, z)$  belongs to the class $C_1$ and is equivalent to  $\mathcal{M}^0_{\Lambda'}(1,q)$, where $\Lambda'=(e_{FS},c_F, 2c_S)$.    
\item[-] if $q=2$ then
again $\partial \mathcal{L}_{F}(x_F)=\partial \mathcal{L}_{FS}(x_F,z)$  
and every film site in $\partial \mathcal{L}_{FS}(x_F,z)$ is bonded with exactly one substrate atom.  
Moreover, since \eqref{firstinterface2} holds and every substrate atom in $\partial \mathcal{L}_S(z)$ is bonded with one film site, this implies that the model  $\mathcal{M}_{\Lambda}(x_F, z)$ belongs to the clas $C_3$   and is equivalent to  $\mathcal{M}^0_{\Lambda}(1,1)$.
\item[-] if $ q\geq 3$ then
$$\partial \mathcal{L}_{FS}(x_F,z)=\{x_F+k\,: \, k \in q \mathbb{Z} \cup (q \mathbb{Z}+1)  \}, $$ or
\begin{eqnarray*}
   \partial \mathcal{L}_{FS}(x_F,z)&=&\{x_F+k\, : \, k \in q \mathbb{Z} \cup (q \mathbb{Z}-1)  \}\\ &=&\{(x_F-1)+k\,:\,k \in q \mathbb{Z} \cup (q \mathbb{Z}+1) \}.
\end{eqnarray*}
Thus,  both situations imply that the model $\mathcal{M}_{\Lambda}(x_F, z)$ belongs to the class $C_4$ with $r=1$ and the model $\mathcal{M}_{\Lambda}(x_F, z)$ is equivalent to  $\mathcal{M}^1_{\Lambda}(q,1)$.
\end{itemize}

It remains to analyze the situation 2b). Fix an arbitrary film site  $a_1 \in \partial \mathcal{L}_{FS}(x_F,z)$ bonded with a substrate atom $s_1 \in \partial \mathcal{L}_S(z)$. Let us denote  the closest film site on the right of $a_1$  that is bonded with a substrate atom by $a_2 \in \partial \mathcal{L}_{FS}(x_F,z)$ and the closest film site on the right of $a_2$  that is bonded with a substrate atom by $a_3 \in \partial \mathcal{L}_{FS}(x_F,z)$. Furthermore, we denote by $s_2, s_3 \in \partial \mathcal{L}_S(z)$ the substrate neighbors of $a_2,a_3 \in \partial \mathcal{L}_{FS}(x_F,z)$, respectively. We also denote by $s_{a_i}$  the projections of $a_i$ on $\partial \mathcal{L}_{S}(z)$ for $i=1,2,3$.
There are two possibilities:
\begin{itemize}
\item[-] the triangle $a_2,s_2,s_{a_2}$ is the translation of the triangle $a_1,s_1,s_{a_1}$. Then ${\rm dist}(s_1,s_2)$ is an integer  (since $a_1,a_2,s_1,s_2$ is a parallelogram), which has to be equal to $q$ (it is  greater or equal to $q$, since it is a multiple of $e_S$ and less or equal to $q$, by the fact that the lattices are $q\tone$ periodic). Therefore, 
$$\partial \mathcal{L}_{FS}(x_F,z)=\{x_F+k: k \in q \mathbb{Z} \},$$
which  implies that the model $\mathcal{M}_{\Lambda}(x_F, z)$ belongs to class $C_1$ and is equivalent with the model $\mathcal{M}^0_{\Lambda}(1,q)$.
\item[-] the triangle $a_2,s_2,s_{a_2}$ is a translated reflection of the triangle $a_1,s_1,s_{a_1}$. Then the triangle $a_3,s_3,s_{a_3}$ is a translation of the triangle $a_1,s_1,s_{a_1}$ (because otherwise it would be a translation of the triangle $a_2,s_2,s_{a_2}$ and ${\rm dist}(s_2,s_3)$ would be equal to $q$ by the previous analysis). Therefore  ${\rm dist}(s_1,s_3)$ is equal to $q$ and to ${\rm dist}(a_1,a_3)$ (it is a multiple of $q$ since $a_1,a_3,s_1,s_3$ is a parallelogram and less or equal to $q$ as the consequence the fact that the lattices are $q\tone$ periodic). If we define  $r={\rm dist}(a_1,a_2)$ we easily conclude that
 \begin{equation} \label{invert1} \partial \mathcal{L}_{FS}(x_F,z)= \{x_F+k: k \in q\mathbb{Z}\cup (q\mathbb{Z}+r) \},  
 \end{equation}
 or
 \begin{eqnarray} \label{invert2}
 \partial \mathcal{L}_{FS}(x_F,z)&=& \{x_F+k: k \in q\mathbb{Z}\cup (q\mathbb{Z}-r) \}\\ \nonumber &=&\{(x_F-r)+k: k \in q\mathbb{Z}\cup (q\mathbb{Z}+r) \}.
 \end{eqnarray}
 Therefore,  we obtain the following implications:
 \begin{itemize}
 \item if $r=\frac{q}{2}$, then  $\mathcal{M}_{\Lambda}(x_F, z)$ belongs to the class $C_3$ and the model is equivalent with the model $\mathcal{M}^0_{\Lambda}(1,\frac{q}{2})$, 
 \item  if $r< \frac{q}{2}$, then $\mathcal{M}_{\Lambda}(x_F, z)$ belongs to the class $C_4$  and the model $\mathcal{M}_{\Lambda}(x_F, z)$ is equivalent to  $\mathcal{M}^1_{\Lambda}(q,r)$,
%  \item  if $r>\frac{q}{2}$, then by replacing $(x_F,z)$ with $(x_F',z')$  for $x_F':=x_F+r$ and $z':=(1,q,r')$ with $r':=q-r$.
  
 \item if $r>\frac{q}{2}$, then by replacing $x_F$ with $x_F'=x_F+r$ and $r$ with $r'=q-r$  in \eqref{invert1} (notice that  $\mathcal{L}_{FS}(x_F',z')=\mathcal{L}_{FS}(x_F,z)$ for $z':=(1,q,r')$)   or  by replacing $r$ by $r':=q-r$ in \eqref{invert2} (notice that $q\mathbb{Z}-r=(q-r)+q\mathbb{Z}$), we conclude that $\mathcal{M}_{\Lambda}(x_F, z)$ belongs to the class $C_4$ and  is equivalent to  $\mathcal{M}^1_{\Lambda}(q,q-r)$.  
\end {itemize}
This concludes the analysis of 2b) and the proof of the assertion.
\end{itemize}  
\end{proof}

In view of Proposition \ref{equivalence_Z1} we can now recover the main results of the manuscripts directly by  \cite[Theorems 2.2-2.4]{PiVe1}  when models $\mathcal{M}_{\Lambda}(x_F, z)\in C_i$ with $i=1,2,3$ and so it is equivalent  with a model of type $\mathcal{M}^0_{\Lambda'}(z')$ for a proper choice of $\Lambda'\in(\Rz^+)^3$ and $z'\in Z_S^0$, and by Theorems \ref{wetting_theorem_Z1}, \ref{connectness_Z1}, and \ref{thm:convergence_minimizers_Z1} for models $\mathcal{M}_{\Lambda}(x_F, z)\in C_4$ that are equivalent to the models treated in Section \ref{sec:M1}, i.e.,  $\mathcal{M}^1_{\Lambda}(z'')$ for a proper choice of  $z''\in Z_S^1$.

We conclude the paper with the list of the proofs of the main results.

\begin{proof}[Proof of Theorem \ref{wetting_theorem}]
The first assertion directly follows from recalling that from \cite[Theorem 2.2]{PiVe1} the wetting condition for $z\in Z^0_S$  is
	\begin{equation}
\begin{cases}
c_{S}\geq 6c_{F} & \text{if either $z \in Z_S^0$ and $q\neq 1$},\\
c_{S}\geq 4c_{F} & \text{if $z \in Z_S^0 $ and $q=1$}, \\
\end{cases}
	\end{equation}
	and from Theorem \ref{wetting_theorem_Z1} the wetting condition for $z\in Z^1_S$  is \eqref{wetting_condition_total}. Furthermore, by Definition \ref{equivalent_model}  the associated configurations to the minimizers of equivalent models are minimizers. 
The second assertion is also a direct consequence of Proposition \ref{equivalence_Z1}  and the second assertions  of \cite[Theorem 2.2]{PiVe1} for $z\in Z^0_S$   and  of Theorem \ref{wetting_theorem_Z1}  for $z\in Z^1_S$.
\end{proof} 

\begin{proof}[Proof of Theorem \ref{connectness}]
We begin by observing that  under the dewetting condition \eqref{dewetting_condition} we obtain the dewetting condition \cite[Eq. (32)]{PiVe1}  related  to models  $\mathcal{M}^0_{\Lambda'}(z')$ with $\Lambda'\in(\Rz^+)^3$ and $z'\in Z_S^0$, i.e.,  the condition 
	\begin{equation}\label{dewetting_condition_Z0} 
\begin{cases}
c_{S}< 6c_{F} & \text{if either $z \in Z_S^0$ and $q\neq 1$},\\
c_{S}< 4c_{F} & \text{if $z \in Z_S^0 $ and $q=1$}, \\
\end{cases}
	\end{equation}, if the model $\mathcal{M}_{\Lambda}(x_F, z)\in C_i$ with $i=1,2,3$, and the dewetting condition \eqref{dewetting_condition_Z1} if  $\mathcal{M}_{\Lambda}(x_F, z)\in C_4$. Therefore, by Proposition \ref{equivalence_Z1} for $i=1,2,3$ we can conclude by employing \cite[Theorem 2.3]{PiVe1} and for $i=4$ by Theorem \ref{connectness_Z1}.
	
%	by Proposition \ref{equivalence_Z1} the model $\mathcal{M}_{\Lambda}(x_F, z)$  is equivalent to $\mathcal{M}^0_{\Lambda'}(z')$ for proper interaction vector $\Lambda'$ and wall vector $z'\in Z_S^0$, or, otherwise, the dewetting condition  related  to $\mathcal{M}^1_{\Lambda'}(z_1',z_0')$ for proper interaction vector $\Lambda'$ and  proper $(z_1',z_0')\in Z^1_S$, if by Proposition \ref{equivalence_Z1} the model $\mathcal{M}_{\Lambda}(x_F, z_1,z_0)$  is equivalent to $\mathcal{M}^1_{\Lambda'}(z_1',z_0')$.  Then, the assertion follows in the former case by \cite[Theorem 2.3]{PiVe1} and in the latter case by Theorem \ref{connectness_Z1}.

\end{proof} 

\begin{proof}[Proof of Theorem \ref{thm:convergence_minimizers}]
As observed in the proof of Theorem \ref{connectness} under the dewetting condition \eqref{dewetting_condition} we obtain the dewetting condition \cite[Eq. (32)]{PiVe1}  related  to models  $\mathcal{M}^0_{\Lambda'}(z')$ with $\Lambda'\in(\Rz^+)^3$ and $z'\in Z_S^0$, i.e.,  the condition 
	\begin{equation}\label{dewetting_condition_Z0} 
\begin{cases}
c_{S}< 6c_{F} & \text{if either $z \in Z_S^0$ and $q\neq 1$},\\
c_{S}< 4c_{F} & \text{if $z \in Z_S^0 $ and $q=1$}, \\
\end{cases}
	\end{equation}, if the model $\mathcal{M}_{\Lambda}(x_F, z)\in C_i$ with $i=1,2,3$, and the dewetting condition \eqref{dewetting_condition_Z1} if  $\mathcal{M}_{\Lambda}(x_F, z)\in C_4$. Therefore, the assertion follows by Proposition \ref{equivalence_Z1} for $i=1,2,3$ from \cite[Theorem 2.4]{PiVe1} and for $i=4$ from Theorem \ref{thm:convergence_minimizers_Z1}.

\end{proof}

 \section*{Acknowledgments}%\settocdepth{part}
 The authors are thankful to the Erwin Schr\"odinger Institute in Vienna, where part of this work was developed during the ESI Joint Mathematics-Physics Symposium ``Modeling of crystalline Interfaces and Thin Film Structures'', and acknowledge the support received from BMBWF through the OeAD-WTZ project HR 08/2020.
P. Piovano  acknowledges support \CCC from the Okinawa Institute of Science and Technology in Japan, from the Wolfgang Pauli Institute (WPI) Vienna, from  the Austrian Science Fund (FWF) through projects P 29681 and TAI 293-N, \BBB and from the Vienna Science and Technology Fund (WWTF) together with the City of Vienna and Berndorf Privatstiftung through the project MA16-005. I. Vel\v{c}i\'c  acknowledges support from the Croatian Science Foundation under grant no. IP-2018-01-8904 (Homdirestproptcm).\EEE


\begin{thebibliography}{99}

%\bibitem{Ahlswede}
%R. Ahlswede and S. L. Bezrukov, Edge isoperimetric theorems
%for integer point arrays, {\it Appl. Math. Lett.},  8 (1995), 2:75--80.

%\bibitem{ABC}
%R. Alicandro, A. Braides, and M. Cicalese, Continuum limits of discrete thin films with superlinear growth densities. {\it Calc. Var.
%Partial Differential Equations}, 33 (2008), 267-297.






\bibitem{AFP}
 \textsc{Ambrosio L.},  \textsc{Fusco N.},  \textsc{Pallara D.}, \emph{Functions of bounded variation and free discontinuity
problems}. Oxford Mathematical Monographs, Clarendon Press, New York, 2000.

%\bibitem{Arroyo}
%M. Arroyo and T.  Belytschko.
% An atomistic-based finite deformation membrane for single layer crystalline films,
%{\it J. Mech. Phys. Solids}, 50 (2002), 9:1941--1977. 
%\bibitem{AlbDes}
 %\textsc{Alberti G.},  \textsc{De Simone A.}, Wetting of rough surfaces: a homogenization approach. \emph{Proc. R. Soc. A}, \textbf{461} (2005), 79--97.


\bibitem{Yeung-et-al12}
 \textsc{Au Yeung Y.},  \textsc{Friesecke G.},  \textsc{Schmidt B.},
Minimizing atomic configurations of short range pair potentials in two dimensions: crystallization in the Wulff-shape,
\emph{Calc. Var. Partial Differ. Equ.}, \textbf{44} (2012), 81--100.

\bibitem{BA} 
 \textsc{Baer E.}, Minimizers of Anisotropic Surface Tensions
Under Gravity: Higher Dimensions
via Symmetrization. \emph{Arch. Rational Mech. Anal.}, \textbf{215} (2015), 531--578. 

 \bibitem{Dodineau-etal} 
 \textsc{Bodineau T.},  \textsc{Ioffe D.},  \textsc{Velenik Y.}, Winterbottom Construction for finite range ferromagnetic models: an $\mathcal{L}^1$-approach. \emph{J. Stat. Phys.},  \textbf{105(1-2)} (2001), 93--131.  



%\bibitem{BL}
%X. Blanc, M. Lewin. The crystallization conjecture: a review, to appear on {\it EMS Surv. Math. Sci} (2015).


%\bibitem{Blanc-LeBris02}
%X. Blanc and C. Le Bris.
%Periodicity of the infinite-volume ground state of a one-dimensional quantum model, {\it Nonlinear Anal.}, 48 (2002), 6:791--803.



%\bibitem{Braides-Cicalese07}
%A. Braides and M. Cicalese.
%Surface energies in nonconvex discrete systems, {\it Math. Models Methods Appl. Sci.}, 17 (2007), 7:985--1037. 

%\bibitem{Braides-et-al07}
%A. Braides,  M. Solci, and E. Vitali.
%A derivation of linear elastic energies from pair-interaction atomistic systems, {\it Netw. Heterog. Media}, 2 (2007), 3:551--567.

%\bibitem{Bollobas}
%B. Bollobas and I. Leader.
%Edge-isoperimetric inequalities in the grid, {\it Combinatorica}, 11 (1991), 4:299--314.


% \bibitem{Brenner}
% D. W. Brenner. Empirical potential for hydrocarbons for use in stimulating the chemical vapor deposition of diamond films, {\it Phys. Rev. B}, 42 (1990), 9458--9471.

%\bibitem{Brenner2}
%D. W. Brenner, O. A. Shenderova, J. A. Harrison, S. J. Stuart, B. Ni, and S. B. Sinnott.
%A second-generation reactive empitical bond order (REBO) potential energy expression for hydrocarbons. {\it J. Phys. Condens. Matter}, 14 (2002), 783--802.



%\bibitem{Cances-LeBris-Maday06} E. Canc\`es, C. Le Bris, Y. Maday. {\it M\'ethodes math\'ematiques en chimie quantique}. Springer, Berlin, 2006.

%\bibitem{Cox-Hill07}
%B. J. Cox and J. M. Hill.
%Exact and approximate geometric parameters for carbon nanotubes incorporating curvature, {\it Carbon}, 45 (2007), 1453--1462.

%\bibitem{Daugherty10}
%S. M. Daugherty. {\it Independent sets and closed-shell independent sets of fullerenes.} Ph.D. Thesis, University of Victoria, 2009.

%\bibitem{Dresselhaus-et-al95}
%M. S. Dresselhaus, G. Dresselhaus, and R. Saito.
%Physics of carbon nanotubes, {\it Carbon}, 33 (1995), 883--891.

%\bibitem{braides-98} A. Braides. Approximation of free-discontinuity problems. Lecture Notes in Mathematics. Springer-Verlag, Berlin, 1998.
%\bibitem{asdmr-17} Almi, Stefano and Dal Maso, Gianni and Toader, Rodica. A lower semicontinuity result for a free discontinuity functional with a boundary term. {\it ournal de Math\'{e}matiques Pures et Appliqu\'{e}es}. 108 (2017), 6: 952--990.  
	
	% \bibitem{CM} 
% \textsc{Caffarelli L.A.},  \textsc{Mellet A.}, Capillary Drops on an Inhomogeneous Surface. \emph{Contemporary Mathematics}. \textbf{446} (2007), 175--201.  

	
	   \bibitem{DP1} 
 \textsc{Davoli E.},  \textsc{Piovano P.}, Derivation of a heteroepitaxial thin-film model. \emph{Interface Free Bound.}, \textbf{22-1} (2020), 1--26.
 
   \bibitem{DP2} 
 \textsc{Davoli E.},  \textsc{Piovano P.}, Analytical validation of the Young-Dupr\'e law for epitaxially-strained thin films. \emph{Math.\ Models Methods Appl.\ Sci.}, \textbf{29-12} (2019), 2183-2223. 
 \FFF
 
       \bibitem{DPS2} 
 \textsc{Davoli E.},  \textsc{Piovano P.},  \textsc{Stefanelli U.}, Wulff shape emergence in graphene. \emph{Math. Models Methods Appl. Sci.}, \textbf{26-12} (2016), 2277-2310. %{\small 0.1142/S0218202516500536}.
 
 
	
    \bibitem{DPS1} 
 \textsc{Davoli E.},  \textsc{Piovano P.},  \textsc{Stefanelli U.}, Sharp $N^{3/4}$ law for the minimizers of the edge-Isoperimetric problem on the triangular lattice. \emph{J. Nonlinear Sci.}, \textbf{27-2} (2017), 627--660. %{\small doi:10.1007/s00332-016-9346-1}.
  
\EEE
 
% \SSS
 %\bibitem{DMNP}
% \textsc{Dayrens F.}, \textsc{Mansou S.}, \textsc{Novaga M.}, \textsc{Pozzeta M.}, Connected perimeter of planar sets, \emph{Advances in Calculus of Variations}, DOI: 10.1515/acv-2019-0050 
% \BBB
 
 \bibitem{DKS} 
     \textsc{Dobrushin R.L.}, \textsc{Koteck\'y}, \textsc{Schlosman S.}, \emph{Wulff construction: a global shape from local interaction}. AMS translations series \textbf{104}, Providence, 1992.      
     

	
%\bibitem{E-Li09}
%W. E and D. Li. On the crystallization of 2{D} hexagonal lattices,  {\it Comm. Math. Phys.}, 286 (2009), 3:1099--1140.

%\bibitem{Flatley-Theil12} L. Flatley, F. Theil. Face-centered cubic crystallization
%of atomistic configurations, preprint (2013).

%\bibitem{Flatley-Taylor-Tarasov-Theil12} L. Flatley, M. Taylor, A. Tarasov, F. Theil. Packing
%twelve spherical caps to maximize tangencies, {\it J. Comput. Appl. Math.},  254 (2013), 220--225.



%\bibitem{Evans} 
%!TEX encoding = UTF-8 Unicode\textsc{Evans L.C.},  \textsc{Gariepy R.F.}, \emph{Measure theory and fine properties of functions}. Studies in Advances Mathematics, CRC Press, Boca Raton, 2015.

\bibitem{F}  
\textsc{Fonseca I.}, The Wulff theorem revisited. \emph{Proc. Roy. Soc. London Ser. A }, \textbf{432} (1991), 125--145.
		
\bibitem{FFLM2}
	      \textsc{Fonseca I.}, \textsc{Fusco N.}, \textsc{Leoni G.}, \textsc{Morini M.}, Equilibrium configurations of epitaxially strained crystalline films: existence and regularity results. \emph{Arch. Ration. Mech. Anal.},  \textbf{186} (2007), 477--537.

\bibitem{FM2}  
  \textsc{Fonseca I.},   \textsc{M\"uller S.}, A uniqueness proof for the Wulff problem, {\it Proc. Edinburgh Math. Soc. } \textbf{119A} (1991), 125--136.
%\bibitem{FM3} 
%  \textsc{Fonseca I.}, \textsc{M\"uller S.}, Relaxation of quasiconvex functionals in {${\rm BV}(\Omega,{\bf R}^p)$} for integrands {$f(x,u,\nabla u)$}, {\it Arch. Rational Mech. Anal.}, \textbf{123} (1993), 1--49.
%\bibitem{Gardner-Radin79}
%C. S. Gardner and C. Radin.
%The infinite-volume ground state of the Lennard-Jones potential, {\it J. Stat. Phys.}, 20 (1979), 6:719--724.

%\bibitem{Grant07} I. P. Grant. {\it Relativistic quantum theory of particles and molecules}, Springer, New York, 2006.

%\bibitem{Grivopoulos09}
%S. Grivopoulos.
%No crystallization to honeycomb or Kagom\'e in free space, {\it J. Phys. A: Math. Theor.}, 42 (2009), 115112.

%\bibitem{Gustafson-Sigal03} S. J. Gustafson, I. M. Sigal. {\it Mathematical concepts of quantum mechanics}, Springer, 2003.

%\bibitem{Harris-et-al11}
%L. Harris, M. Taylor, F. Theil, and A. Tarasov.
%Packing twelve spherical caps to maximize tangencies, 2011.

%\bibitem{Harris11}
%L. C. Harris.
%{\it Face-centered cubic structures in energy-minimizing atomistic configurations}, PhD Thesis, Warwick University, 2011

%\bibitem{Harris-Theil12}
%L. C. Harris and F. Theil.
%Face-centered cubic crystallization of atomistic configurations. In preparation, 2012.

%\bibitem{Hamrick-Radin79} G. C. Hamrick and C. Radin. The symmetry of ground states under perturbation, {\it J. Stat. Phys.}, 21 (1979), 5:601--607.

%\bibitem{Harary}
%F. Harary and H. Harborth, Extremal animals, {\it J. Combin. Inform. and System Sci.},
%1 (1976), 1-8.


%\bibitem{Harborth74bis} H. Harborth. \"Uber Primteiler von Stirlingschen
%Zahlen zweiter Art, {\it Elem. Math.}, 29 (1974), 129--131.

%\bibitem{Harborth74} H. Harborth, { Solution to Problem 664A}, {\it Elem. Math.}, 29 (1974), 14--15.

%\bibitem{OnnovanGaans}
%O. van Gaans. Probability measures on on metric spaces. \url{https://www.math.leidenuniv.nl/~vangaans/jancol1.pdf} 


%\bibitem{Cherednichenko-etal}
% \textsc{Guenneau S.},  \textsc{Craster R.},  \textsc{Antonakakis T.},  \textsc{Cherednichenko K.},  \textsc{Cooper S.}, 
%Homogenization Techniques for Periodic Structures. \emph{Theory and Numeric Applications} AMU (PUP), (2012), 11.1-11.31. 

\FFF
        \bibitem{Heitmann-Radin80}
     \textsc{Heitmann R.},      \textsc{Radin C.}, Ground states for sticky disks, {\it J. Stat. Phys.}, \textbf{22} (1980), 281--287.
       

\bibitem{ISc}
        \textsc{Ioffe D.},  \textsc{Schonmann R.}, Dobrushin-Koteck\'y-Shlosman theory up to the critical temperature. \emph{Comm. Math. Phys.}, \textbf{199}, (1998) 117--167.
    
    \bibitem{srolovitz1}
\textsc{Jiang, W.}, \textsc{Wang, Y.}, \textsc{Zhao, Q. }, \textsc{Srolovitz, D. J.}, \textsc{Bao, W.}, Solid-state dewetting and island morphologies in strongly anisotropic materials, \emph{Scripta Materialia}, \textbf{115} (2016), 123--127. 

\bibitem{srolovitz2}
\textsc{Jiang, W.}, \textsc{Wang, Y.}, \textsc{Zhao, Q. }, \textsc{Srolovitz, D. J.}, \textsc{Bao, W.}, Stable Equilibria of Anisotropic Particles on Substrates: A Generalized Winterbottom Construction,\emph{SIAM Journal on Applied Mathematics} \textbf{77} (6) (2017), 2093--2118. 
%\bibitem{kohn1}
%\textsc{Kohn, R.V.}, \textsc{Sternberg, P.}, Local minimisers and singular perturbations, \emph{Proceedings of the Royal Society of Edinburgh Section A: Mathematics} \textbf{111}, Issue 1--2 (1989), 69--84.     

\EEE
			\bibitem{KP} 
	     \textsc{Koteck\'y R.},       \textsc{Pfister C.}. Equilibrium shapes of crystals attached to walls, {\it Jour. Stat. Phys.} \textbf{76} (1994), 419--446.

        \bibitem{KrP}
  \textsc{Kreutz L.},  \textsc{Piovano P.}, Microscopic validation of a variational model of epitaxially strained crystalline films, Submitted (2019).
  

%\bibitem{Kamatgalimov-Kovalenko10}
%A. R. Kamatgalimov and V. I. Kovalenko.
%Deformation and thermodynamic instability of a $C_{84}$ fullerene cage, {\it Russ. J. Phys. Chem. A}, 84 (2010), 4L721--726.

 
%\bibitem{Kroto87}
%H. W. Kroto.
%The stability of the fullerenes $D_n$, with $n=24, 28, 32, 36, 50, 60$ and $70$, {\it Nature}, 329 (1987), 529-531.

%\bibitem{LeBris-Lions05} C. Le Bris, P.-L. Lions. From atoms to crystals: a mathematical journey, {\it Bull. AMS}, 42 (2005), 291-363.

%\bibitem{Lin-et-al10}
%F. Lin, E. S\o rensen, C. Kallin, and  J. Berlinsky. $C_{20}$, the smallest fullerene, in {\it Handbook of Nanophysics: Clusters and Fullerenes} (K. D. Sattler, ed.), Taylor \& Francis, CRC Press, 2010.

%\bibitem{Mainini-Stefanelli12}
%E. Mainini and U. Stefanelli. Crystallization in carbon nanostructures, {\it Comm. Math. Phys.}, to appear, 2014. 


%\bibitem{Maggi} 
 %     \textsc{Maggi F.}, \emph{Sets of finite perimeter and geometric variational problems. An introduction to geometric measure theory}. Cambridge studies in advanced mathematics \textbf{135},  Cambridge University Press, 2012.


%\bibitem{MMPS}
%E. Mainini, H. Murakawa, P. Piovano, and U. Stefanelli. A numerical investigation on carbon-nanotube geometries, Preprint IMATI-CNR, 4PV14/0/0 (2014).



   \bibitem{MPSS} 
  \textsc{Mainini E.},  \textsc{Piovano P.},  \textsc{Schmidt B.}, \textsc{Stefanelli U.}, $N^{3/4}$ law in the cubic lattice.  \emph{J. Stat. Phys.}, \textbf{176-6} (2019), 480--1499. % {\small doi:10.1007/s10955-019-02350-z}.

  
  \bibitem{MPS} 
 \textsc{Mainini E.},  \textsc{Piovano P.},  \textsc{Stefanelli U.}, Finite crystallization in the square lattice. \emph{Nonlinearity}, \textbf{27} (2014), 717--737. % {\small doi:10.1088/0951-7715/27/4/717}
 
 

 

\bibitem{MPS2} 
 \textsc{Mainini E.},  \textsc{Piovano P.},  \textsc{Stefanelli U.}, \emph{Crystalline and isoperimetric square configurations}. Proc. Appl. Math. Mech. \textbf{14} (2014), 1045--1048. %{\small doi:10.1002/pamm.201410494}. 
 
   
  \bibitem{MS} 
 \textsc{Mainini E.}, \textsc{Schmidt B.},  Maximal fluctuations around the Wulff shape for edge-isoperimetric sets in $\mathbb{Z}^d$: a sharp scaling law. \FFF \emph{Comm. Math. Phys.},  in press (2020), \EEE {\small https://arxiv.org/pdf/2003.01679.pdf}. 

%\bibitem{Mielke12}
%A. Mielke. Private communication, 2012.

%\bibitem{Mueller93}
%S. M\"uller.
%Singular perturbations as a selection criterion for periodic minimizing sequences,
%{\it Calc. Var. Partial Differential Equations},  1 (1993), 2:169--204.



     \bibitem{PV1}     
      \textsc{Pfister C.E.},   \textsc{Velenik Y.}, Mathematical theory of the wetting phenomenon in the 2D Ising model. \emph{Helv. Phys. Acta}, \textbf{69} (1996), 949--973.

     
\bibitem{PV2}
      \textsc{Pfister C.E.},   \textsc{Velenik Y.}, Large deviations and continuous limit in the 2D Ising model. \emph{Prob. Th. Rel. Fields} \textbf{109} (1997), 435--506.
      
      
\bibitem{PiVe1} 
\CCC \textsc{Piovano P.},  \textsc{Vel\v{c}i\'c I.}, \emph{Microscopical Justification of Solid-State Wetting and Dewetting}. \emph{Submitted} (2021). %{\small doi:10.1002/pamm.201410494}. 
 
\BBB


%\bibitem{Radin81} 
%C. Radin.  The ground state for soft disks, {\it J. Stat. Phys.}, 26 (1981), 2:365--373.

%\bibitem{Radin83}
%C. Radin. Classical ground states in one dimension, {\it J. Stat. Phys.}, 35 (1983), 1-2:109--117.

%\bibitem{Radin86} C. Radin. Crystals and quasicrystals: a continuum model, {\it Comm. Math. Phys.}, 105 (1986), 385--390.

%\bibitem{Radin87}
%C. Radin. Low temperature and the origin of crystalline symmetry. {\it Int. J. Mod. Phys. B}, 1 (1987), 5-6:1157–
%1191.

%\bibitem{Radin-Schulmann83} C. Radin and L. S. Schulmann. Periodicity of classical ground states, {\it Phys. Rev. Lett.}, 51 (1983), 8:621--622. 

\bibitem{Schmidt}
      \textsc{Schmidt B.}, Ground states of the 2D sticky disc model: fine properties and $N^{3/4}$ law for the deviation from the asymptotic Wulff-shape. {\it J. Stat. Phys.}, \textbf{153}  (2013), 727--738.

\bibitem{S2}
	      \textsc{Spencer B.J.}, Asymptotic derivation of the glued-wetting-layer model and the contact-angle condition for Stranski-Krastanow islands. \emph{Phys.  Rev. B}, \textbf{59} (1999), 2011--2017.	 	

\bibitem{spencer1997equilibrium}
\textsc{Spencer B.J.},    \textsc{Tersoff J.}. Equilibrium shapes and properties of epitaxially strained islands. {\em Physical Review Letters}, \textbf{79-(24)} (1997), 4858.

%\bibitem{Ventevogel78}
%W. J. Ventevogel.
%On the configuration of a one-dimensional system of interacting particles with minimum
%potential energy per particle, {\it Phys. A.}, 92 (1978), 3-4:343--361.


%\bibitem{Ventevogel-Nijboer79bis}
%W. J. Ventevogel and B. R. A. Nijboer. 
%On the configuration of systems of interacting particle with minimum
%potential energy per particle, {\it Phys. A.}, 99, (1979), 3:565--580.



%\bibitem{Schein-Friedrich08}
% S. Schein and T. Friedrich. A geometric constraint, the head-to-tail exclusion rule, may be the basis for the isolated-pentagon rule for fullerenes with more than 60 vertices. {\it Proc. Nat. Acad. Sci. U.S.A.}, 105 (2008), 19142--19147.


%\bibitem{Stillinger-Weber85} F. H. Stillinger and T. A. Weber. Computer simulation of local order in condensed phases of silicon, {\it Phys. Rev. B}, 8 (1985), 5262--5271.

      
%\bibitem{T}
 %      \textsc{Taylor J.E.},  Existence and structure of solutions to a class of non elliptic variational problems. \emph{Sympos. Math.}, \textbf{14} (1974), 499--508.	          
        
   
      
%\bibitem{T2}
 %       \textsc{Taylor J.E.}, Unique structure of solutions to a class of non elliptic variational problems.  \emph{Proc. Sympos. Pure Math.}, \textbf{27} (1975), 419--427.	      
          

 
%\bibitem{Tersoff}
%J. Tersoff. New empirical approach for the structure and energy of covalent systems.  {\it Phys. Rev. B}, 37 (1988), 6991--7000.

   
%\bibitem{Theil06}
%F. Theil. A proof of crystallization in two dimensions, {\it Comm. Math. Phys.}, 262 (2006), 1:209--236.


%\bibitem{Theil11}
%F. Theil. Suface energies in a two-dimensional mass-spring model for crystals, {\it  ESAIM Math. Model. Numer. Anal.}, 45 (2011), 873--899.


%\bibitem{Ventevogel-Nijboer79} W. J. Ventevogel and B. R. A. Nijboer. On the configuration of systems of interacting particle with minimum potential energy per particle, {\it Phys. A.}, 98, (1979), 1-2:274--288.


%\bibitem{Wagner83} H. J. Wagner. Crystallinity in Two Dimensions: A Note on a Paper of C. Radin, {\it J. Stat. Phys.},
 %33  (1983), 3:523-526,
 
 	\bibitem{Winterbottom} 
      \textsc{Winterbottom W.L.}, Equilibrium shape of a small particle in contact with a foreign substrate. {\it Acta Metallurgica}, \textbf{15} (1967), 303--310.

 
%\bibitem{Wu}
%J. Wu, K. C. Hwang, Y. Huang. An atomistic-based finite-derformation shell theory for single-wall carbon nanotubes, {\it J. Math. Phys. Solids}, 56 (2008), 279--292.

		
		\bibitem{Wulff01} 
		      \textsc{Wulff G.}, Zur Frage der Geschwindigkeit des Wastums und der Aufl\"osung der Kristallflachen. Krystallographie und Mineralogie. {\it Z. Kristallner.}, \textbf{34} (1901), 449--530.

%\bibitem{Yedder-Blanc-LeBris03} A. B. Yedder, X. Blanc, C. Le Bris.  A numerical investigation of the 2-dimensional crystal problem, preprint (2003).











\end{thebibliography}
\end{document}